\documentclass[a4paper,12pt]{article}
\usepackage[top=2cm, bottom=2cm, outer=2cm, inner=2cm, headsep=14pt]{geometry}
\usepackage[dvips]{epsfig,graphicx}
\usepackage{amsmath,amsfonts}
\usepackage{color}
\usepackage{url}

\usepackage{tikz}
\usepackage{enumitem}
\usepackage{mleftright} 

\usepackage{float}
\usepackage{multirow}
\usepackage{amssymb}
\usepackage{amscd}

\newenvironment{proof}{{\noindent\it Proof. }}{\nopagebreak\hspace*{0.5cm}\hfill$\hbox{\rule{3pt}{6pt}}$\smallskip}

\newcommand{\RR}{{\mathbb R}}
\newcommand{\N}{{\mathcal N}}
\newcommand{\A}{{\mathcal A}}
\newcommand{\R}{{\mathcal R}}
\newcommand{\F}{{\mathcal F}}
\newcommand{\C}{{\mathcal C}}
\newcommand{\D}{{\mathcal D}}
\newcommand{\M}{{\mathcal M}}
\newcommand{\NN}{{\mathbb N}}
\newcommand{\ZZ}{{\mathbb Z}}
\newcommand{\CC}{{\mathbb C}}
\newcommand{\FF}{{\mathbb C}}

\newcommand{\G}{\Gamma}

\def\ww{{\boldsymbol w}}
\def\tt{{\boldsymbol t}}
\def\pp{{\boldsymbol p}}
\def\jj{{\boldsymbol j}}

\def\0{{\boldsymbol 0}}
\newcommand{\W}{{\cal W}}
\newcommand{\V}{{\cal V}}
\newcommand{\T}{{\cal T}}

\newcommand{\Mat}{\hbox{\rm Mat}}

\newcommand{\spec}{\mbox{\rm sp}}

\def\dist{{\mbox{\rm dist}}}
\def\dgr{{\mbox{\rm dgr}}}
\def\rank{{\mbox{\rm rank}}}
\def\summ{{\mbox{\rm sum}}}
\def\trace{{\mbox{\rm tr}}}
\def\P{{\cal P}}

\newtheorem{theorem}{Theorem}[section]
\newtheorem{lemma}[theorem]{Lemma}
\newtheorem{corollary}[theorem]{Corollary}
\newtheorem{proposition}[theorem]{Proposition}
\newtheorem{definition}[theorem]{Definition}

\newtheorem{remark}[theorem]{Remark}
\newtheorem{example}[theorem]{Example}

\newtheorem{problem}{Problem}[section]
\newtheorem{algorithm}[theorem]{Algorithm}
\newtheorem{comment}[theorem]{Comment}

\newfont\fiverm{cmr5}
\def\eeq{\end{equation}}
\def\lbeq#1{\begin{equation} \label{#1}}

\newcommand{\ds}{\displaystyle}
\def\span{{\mbox{\rm span}}}
\def\im{{\mbox{\rm im}}}
\def\ol{\overline}

\title{On symmetric association schemes \\
and associated quotient-polynomial graphs\footnote{This research has been partially supported by
AGAUR from the Catalan Government under project 2017SGR1087 and by MICINN from the Spanish Government under project PGC2018-095471-B-I00. The second author acknowledges the financial support from the Slovenian Research Agency (research program P1-0285 and research project J1-1695).}}

\author{M. A. Fiol\\
	{\small Departament de Matem{\`a}tiques} \\
   	{\small     Universitat Polit{\' e}cnica de Catalunya}\\
    	{\small    Barcelona Graduate School of Mathematics} \\
      	{\small  Catalonia, Spain} \\
       	{\small {miguel.angel.fiol@upc.edu}} \and
Safet Penji{\'c}\\
      	{\small  University of Primorska} \\
     	{\small  Muzejski trg 2 }\\
    	{\small    6000 Koper, Slovenia} \\
     	{\small   {Safet.Penjic@iam.upr.si}}}

\begin{document}

\maketitle

\begin{abstract}
Let $\G$ denote an undirected, connected, regular graph with vertex set $X$, adjacency matrix $A$, and ${d+1}$ distinct eigenvalues. Let $\A=\A(\G)$ denote the subalgebra of $\Mat_X(\CC)$ generated by $A$. We refer to $\A$ as the {\it adjacency algebra} of $\G$. In this paper we investigate algebraic and combinatorial structure of $\G$ for which the adjacency algebra $\A$ is closed under Hadamard multiplication. In particular, under this simple assumption, we show the following: (i) $\A$ has a standard basis $\{I,F_1,\ldots,F_d\}$; (ii) for every vertex there exists identical distance-faithful intersection diagram of $\G$ with $d+1$ cells; (iii) the  graph $\G$ is quotient-polynomial; and (iv) if we pick $F\in \{I,F_1,\ldots,F_d\}$ then $F$ has $d+1$ distinct eigenvalues if and only if $\span\{I,F_1,\ldots,F_d\}=\span\{I,F,\ldots,F^d\}$. We describe the combinatorial structure of quotient-polynomial graphs with diameter $2$ and $4$ distinct eigenvalues. As a consequence of the technique from the paper we give an algorithm which computes the number of distinct eigenvalues of any Hermitian matrix  using only elementary operations.  When such a matrix is the adjacency matrix of a graph $\G$, a simple variation of the algorithm allow us to decide wheter $\G$ is distance-regular or not. In this context,  we also propose an algorithm to find which distance-$i$ matrices are polynomial in $A$, giving also these polynomials.
\end{abstract}


\smallskip
{\small

}


\section{Introduction}
\label{1a}


A matrix algebra is a vector space of matrices which is closed with respect to matrix multiplication. Let $X$ denote a finite set and $\Mat_X(\CC)$ the set of complex square matrices with rows and columns indexed by $X$ (or full algebra denoted by $\CC_{|X|}$). The subalgebras of  $\Mat_{X}(\CC)$ that are closed under (elementwise) Hadamard multiplication, and containing the all-ones matrix $J$, are known as coherent algebras. The concept was developed independently by Weisfeiler and  Lehman in \cite{WL} and by  Higman in \cite{Hcc, Hca}. A good introduction to the topic may be found in \cite{KMMZ}. In the literature, a rich theory has been built up around this concept, and much more can be found in \cite{IJR, JKM, KCG, Swa, SAD, ST, SS, XB}. It is well known that every coherent algebra $\C$ is semisimple (see, for example, \cite[Section 2]{FS}) and that has a standard basis $\{N_0,N_1,\ldots,N_r\}$ consisting of the primitive idempotents of $\C$ viewed as a subalgebra of $\Mat_X(\CC)$ with respect to Hadamard multiplication (see \cite{Hca}). Each {\it basis matrix} $N_i$ of a coherent algebra $\C=\langle N_0, N_1,\ldots, N_r\rangle$ can be regarded as the adjacency matrix $A=A(\G_i)$ of a graph $\G_i=(X,R_i)$. Then $\G_i$ and $R_i$ are called a {\it basis graph} and a {\it basis relation}, respectively, of the coherent algebra $\C$. The basis relations of a coherent algebra give rise to a {\it coherent configuration} in the sense of \cite{Hcc}.

A special subfamily of coherent configurations are commutative association schemes also known as homogeneous coherent configurations \cite{EP}. Let $\R=\{R_0,R_1,\ldots,R_n\}$ denote a set of nonempty subsets of $X\times X$. For each $i$, let $A_i\in\Mat_X(\CC)$ denote the adjacency matrix of the (in general, directed) graph $(X,R_i)$ . The pair $(X,\R)$ is an {\it association scheme} with $n$ classes if
\begin{enumerate}
[label=\rm(AS\arabic*)]
\item
$A_0=I$, the identity matrix.
\item
$\ds{\sum_{i=0}^n A_i=J}$, the all-ones matrix.
\item
${A_i}^\top\in\{A_0,A_1,\ldots,A_n\}$ for $0\le i\le n$.
\item
$A_iA_j$ is a linear combination of $A_0,A_1,\ldots,A_n$ for $0\le i,j\le n$.
\end{enumerate}
By (AS1) and (AS4) the vector space $\M$ spanned by the set $\{A_0,A_1,\ldots,A_n\}$ is an algebra; this is the {\it Bose-Mesner algebra} of $(X,\R)$. We say that $(X,\R)$ is {\it commutative} if $\M$ is commutative, and that $(X,\R)$ is {\it symmetric} if the  matrices $A_i$ are symmetric. A symmetric association scheme is commutative. The concept of (symmetric) association schemes can also be viewed as a purely combinatorial generalization of the concept of finite transitive permutation groups (famously said as a ``group theory without groups''\cite{BI}). The Bose-Mesner algebra was introduced in \cite{BM}, and the monumental thesis of  Delsarte \cite{Daa} proclaimed the importance of commutative association schemes as a unifying framework for coding theory and design theory. There are a number of excellent articles and textbooks on the theory of (commutative) association schemes and Delsarte's theory; see, for instance,  \cite{BRA, BCN, DL, FT, HS, MT}. The following are some of the books which include accounts on commutative association schemes: \cite{CL, Gac, MWS, LW}. As an example of commutative association scheme, let $\G$ denote a distance regular graph of diameter $D$. It is well known (and not hard to prove it) that the vector space spanned by distance-$i$ matrices $\{A_0,A_1,\ldots,A_D\}$ of $\G$, is closed under both ordinary multiplication $(A,B)\mapsto AB$ and Hadamard multiplication $(A,B)\mapsto A\circ B$ (see, for example, \cite[Chapter III]{BI} or \cite[Chapter 4]{BCN}). This is one of the main reasons why the theory of distance-regular graphs is so rich in the study of algebraic and combinatorial structures.

In this paper we consider the following problem (we always assume that our graphs are finite, simple, and connected; see Section \ref{2a} for formal definitions).

\begin{problem}
{\label{1b}}
Let $\G$ denote a regular graph with vertex set $X$. Using the algebraic or combinatorial structure of $\G$, is it possible to find a set $\{F_0,F_1,\ldots,F_d\}\subset\Mat_{X}(\CC)$ $($for some $d\in\NN$\/$)$ such that the following hold?
\begin{enumerate}
[label=\rm(\roman*)]
\item
$F_i$'s $(0\le i\le d)$ are nonzero $(0,1)$-matrices, such that $F_i\circ F_j=\delta_{ij}F_i$ $(0\le i,j\le d)$.
\item
There exist $\Omega\subset\{0,1,\ldots,d\}$ such that
$\sum_{\alpha\in\Omega} F_{\alpha}=I$, the identity matrix.
\item
$\ds{\sum_{i=0}^d F_i=J}$, the all-ones matrix.
\item
For each $i$ $(0\le i\le d)$, ${F_i}^\top\in\{F_0,F_1,\ldots,F_d\}$.
\item
The vector space $\span\{F_0,F_1,\ldots,F_d\}$ is closed under both ordinary multiplication $(A,B)\rightarrow AB$ and Hadamard multiplication $(A,B)\rightarrow A\circ B$.
\item
$\span\{F_0,F=F_1,\ldots,F_d\}=\{p(F)\,|\,p\in\RR[t]\}$.
\end{enumerate}
\end{problem}
\medskip

A basis $\{F_0,F_1,\ldots,F_d\}$ of some subalgebra $\C$ of the matrix algebra $\Mat_{X}(\CC)$ for which properties (i)--(v) of Problem \ref{1b} hold, is known as the {\it standard basis} of $\C$. As consequence of property (v) we have that $F_iF_j$ is a linear combination of $F_0,F_1,\ldots,F_d$ for $0\le i,j\le d$, that is, there exist {\it intersection numbers} $p^h_{ij}$ $(0\le i,j,h\le d)$ such that $\ds{F_iF_j=\sum_{i=0}^d p^h_{ij} F_h }$.

Our main results are the Theorems \ref{1c}, \ref{1g}, \ref{1f}, \ref{1d}, \ref{1h}, \ref{1e} and Algorithms \ref{3d}, \ref{3d2}, and \ref{gi}.

\begin{theorem}
\label{1c}
Let $\G$ denote a regular graph with $d+1$ distinct eigenvalues. If the vector space $\A=\span\{I,A,\ldots,A^d\}$ is closed under Hadamard multiplication $(A,B)\rightarrow A\circ B$, then there exists a unique basis $\{F_0,F_1,\ldots,F_d\}$ of $\A$ such that the following hold.
\begin{enumerate}[label=\rm(\roman*)]
\item
$F_i$'s $(0\le i\le d)$ are nonzero $(0,1)$-matrices, such that $F_i\circ F_j=\delta_{ij}F_i$ $(0\le i,j\le d)$.
\item
There exist $m\in\{0,1,\ldots,d\}$ such that $F_m=I$, the identity matrix.
\item
$\ds{\sum_{i=0}^d F_i=J}$, the all-ones matrix.
\item
${F_i}^\top=F_i$ 
$(0\le i\le d)$.
\item
$F_iF_j$ is a linear combination of $F_0,F_1,\ldots,F_d$ for $0\le i,j\le d$.
\end{enumerate}
\end{theorem}

Note the similarity between \cite[Theorem 2.6.1]{BCN} and our Theorem \ref{1c}. As a consequence of Theorem \ref{1c}, we see that if the adjacency algebra $\A$ of $\G$ is closed under Hadamard multiplication then it produces a symmetric association scheme. The property (ii) of Theorem \ref{1c} tell us that if we want to get property (ii) of Problem \ref{1b}, for $|\Omega|>1$, we should consider a directed graph $\G$. By Theorem \ref{1c}(iv), we also need a directed graph to get non-symmetric $F_i$'s. Using the technique from the proof of Theorem \ref{1c}, in Subsection \ref{3c}, we give an algorithm which yields the number of distinct eigenvalues of $A$ without computing them.

The next question we want to answer is what is the combinatorial structure of $\G$ for which the vector space $\A=\span\{I,A,\ldots,A^d\}$ is closed under Hadamard multiplication.

\begin{theorem}
\label{1g}
Let $\G$ denote a regular graph with $d+1$ distinct eigenvalues. If the vector space $\A=\span\{I,A,\ldots,A^d\}$ is closed under Hadamard multiplication, then, for every vertex $x$, there exists an $x$-distance-faithful intersection diagram with $d+1$ cells. Moreover, this intersection diagram is the same around every vertex.
\end{theorem}

For the converse of Theorem \ref{1g}, see Theorem \ref{1f}. The first author in \cite{FQ} defined quotient-polynomial graphs, as graphs for which the adjacency matrices of walk-regular partition belong to adjacency algebra $\A$. In the same paper some combinatorial properties of these graphs were studied. In Section \ref{go} we recall some old, and prove some new, properties of quotient-polynomial graphs. We also consider graphs which have the same distance-faithful intersection diagram around every vertex, and we propose a method for deciding is their distance-$i$ matrix $A_i$ is polynomial in $A$.

\begin{theorem}
\label{1f}
Let $\G$ denote a graph with vertex set $X$, $x$-distance-faithful intersection diagram $\pi_x$, and assume that $\pi_x$ has $r+1$ cells $\P_i$ with $\P_0=\{x\}$: $\pi_x=\{\P_0,\P_1,\ldots,\P_r\}$. Let $w_{ij}$ denote the number of $i$-walks $(0\le i\le r)$ from $y$ to $x$ for any $y\in\P_j$ $(0\le j\le r)$. Let $P=[w_{ij}]_{0\le i,j\le r}$ denote $(r+1)\times(r+1)$ matrix with entries $w_{ij}$. If $\G$ has the same $x$-distance-faithful intersection diagram around every $x\in X$ then $\G$ has exactly $\rank(P)$ distinct eigenvalues. Moreover, if $\rank(P)=r+1$ then $\G$ is a quotient-polynomial graph.
\end{theorem}

\noindent
For the moment assume that $\G$ is a distance-regular graph with diameter $D$. Note that intersection diagram of a distance partition around $x$ of $\G$ has $D+1$ cells, and is the same for every $x\in X$ (also it is $x$-distance-faithful). So as an immediate corollary of Theorem \ref{1f}, the number of distinct eigenvalues of a distance-regular graph $\G$ is $\le D+1$. Also note that the nonnegative integer $w_{ij}$ from the Theorem \ref{1f} can be computed from the $x$-distance-faithful intersection diagram.

In Theorem \ref{1d} we establish a connection between the structure of $\G$ and Problem \ref{1b}.

\begin{theorem}
\label{1d}
Let $\G$ denote a regular graph with $d+1$ distinct eigenvalues. Then, the vector space $\A=\span\{I,A,\ldots,A^d\}$ is closed under Hadamard multiplication if and only if $\G$ is a quotient-polynomial graph.
\end{theorem}

As a corollary of Theorem \ref{1d}, if the number of distinct entries of $A^d$ is greater than $d+1$, then the adjacency algebra $\A$ is not closed under Hadamard multiplication (see also Section \ref{go}).

In Theorem \ref{1h} we consider quotient-polynomial graphs with diameter $2$, and $4$ distinct eigenvalues. Quotient-polynomial graphs with diameter $2$, and $3$ distinct eigenvalues are known as strongly-regular graphs.

\begin{theorem}
\label{1h}
Let $\G$ denote a regular connected graph with diameter $2$ and $4$ distinct eigenvalues. Then the vector space $\A=\span\{I,A,A^2,A^3\}$ is closed under Hadamard multiplication if and only if either {\rm (i)} or {\rm (ii)} bellow hold.
\begin{itemize}
\item[{\rm (i)}]
Any two nonadjacent vertices have a constant number of common neighbours, and the number of common neighbours of any two adjacent vertices takes precisely two values.
\item[{\rm (ii)}]
Any two adjacent vertices have a constant number of common neighbours, and the number of common neighbours of any two nonadjacent vertices takes precisely two values.
\end{itemize}
\end{theorem}

\noindent
Note the similarity between \cite[Theorem 5.1]{ED} and Theorem \ref{1h}.

\medskip
To get property (vi) of Problem \ref{1b}, we have Theorem \ref{1e}.

\begin{theorem}
\label{1e}
Let $\G$ denote a quotient-polynomial graph with $d+1$ distinct eigenvalues, and let $\{I,F_1,\ldots,F_d\}$ denote the standard basis of the adjacency algebra $\A$. Pick $F\in\{F_0,F_1,\ldots,F_d\}$. Then $F$ has $d+1$ distinct eigenvalues if and only if $\span\{F_0,F_1,\ldots,F_d\}=\span\{I,F,\ldots,F^d\}$.
\end{theorem}

Note that Theorems \ref{1c}, \ref{1d} and \ref{1e} give a solution of Problem \ref{1b}.

The paper is organized as follows: in Section \ref{2a} we recall some  notation and definitions.  In Section \ref{3a} we prove Theorem \ref{1c}, in Subsection \ref{3c} we give an algorithm which gives the number of different eigenvalues of a Hermitian matrix without computing them, and in Subsection \ref{3c2} we propose a simple algorithm to check distance-regularity.  In Section \ref{7a} we prove Theorem \ref{1g}. In Section \ref{go} we re-prove some old and obtain some new results about quotient-polynomial graphs, and we prove Theorem \ref{1f}. In Subsection \ref{gn} we give an algorithm which computes the polynomial $p_i(t)$ so that $A_i=p_i(A)$ (if such polynomial exists). In Section \ref{4a} we prove Theorems \ref{1d} and  Theorem \ref{1h}. In Section \ref{5a} we prove Theorem \ref{1e}. Finally, in the last Section \ref{6a} we propose some open problems.


\section{Definitions and preliminaries}
\label{2a}
A {\it graph} (or an {\it undirected graph}) $\G$ is a pair $(X, R)$, where $X$ is a nonempty set and $R$ is a collection of two element subsets of $X$. The elements of $X$ are called the {\it vertices} of $\G$, and the elements of $R$ are called the {\it edges} of $\G$. When $xy\in R$, we say that vertices $x$ and $y$ are {\it adjacent}, or that $x$ and $y$ are {\it neighbors}.
A graph is {\it finite} if both its vertex set and edge set are finite. If we allow for an edge to start and to end at the same vertex, then an edge with identical ends is called a {\it loop}, and a graph is {\it simple} if it has no loops and no two of its edges join the same pair of vertices.
For any two vertices $x, y \in X$, a {\it walk} of length $h$ from $x$ to $y$ is a sequence $x_0,x_1,x_2,\ldots,x_h$ $(x_i\in X,\, 0\le i\le h)$ such that $x_0 = x$, $x_h = y$, and $x_i$ is adjacent to $x_{i+1}$ $(0\le i\le h-1)$. We say that $\G$ is {\it connected} if for any $x, y\in X$, there is a walk from $x$ to $y$. From now on, we assume that $\G$ is finite, simple and connected.

For any $x, y\in X$, the {\it distance} between $x$ and $y$, denoted $\dist(x, y)$, is the length of the shortest walk from $x$ to $y$. The {\it diameter} $D = D(\G)$ is defined to be
$$
D = \max\{\dist(u,v)\,|\,u, v\in X\}.
$$

Let $\G = (X, R)$ be a graph with diameter $D$. For a vertex $x\in X$ and any non-negative integer $h$ not exceeding $D$, let $\G_h(x)$ denote the subset of vertices in $X$ that are at distance $h$ from $x$. Let $\G(x)=\G_1(x)$ and $\G_{-1}(x) = \G_{D+1}(x) := \emptyset$. For any two vertices $x$ and $y$ in $X$ at distance $h$, let
\begin{eqnarray*}
c_h(x,y) & := & \G_{h-1}(x)\cap\G(y),\\
a_h(x,y) & := & \G_{h}(x)\cap\G(y),\\
b_h(x,y) & := & \G_{h+1}(x)\cap\G(y).
\end{eqnarray*}
\noindent
We say $\G$ is {\it regular with valency} $k$, or {$k$-regular}, if each vertex in $\G$ has exactly $k$ neighbours. A graph $\G$ is called {\it distance-regular} if there are integers $b_i$, $c_i$ $(0\le i\le D)$ which satisfy $c_i = |c_i(x, y)|$ and $b_i = |b_i (x, y)|$ for any two vertices $x$ and $y$ in $X$ at distance $i$. Clearly such a graph is regular of valency $k := b_0$, $b_D=c_0=0$, $c_1=1$ and
$$
a_i := |a_i(x,y)| = k - b_i - c_i\quad (0\le i\le D)
$$
\noindent
is the number of neighbours of $y$ in $\G_i(x)$ for $x, y\in X$ $(\dist(x,y)=i)$. For more information about distance-regular graphs, we refer a reader to \cite{DKT}. Some excellent articles that contains algebraic approach to the theory of distance-regular graphs are \cite{ADF, ADF2, MF, FGG, AN, T2}.
An $(n,k,\lambda,\mu)$ {\it strongly-regular graph} is a distance-regular graph $\G=(X,R)$ of diameter $2$ with $|X|=n$, $b_0=k$, $a_1=\lambda$ and $c_2=\mu$.


A {\it partition around} $x$ of $\G$, is a partition $\{\P_0=\{x\},\P_1,\ldots,\P_s\}$ of the vertex set $X$, where $s$ is a positive integer.
The {\it eccentricity} of $x$, denoted by $\varepsilon(x)$, is the maximum distance between $x$ and any other vertex $y$ of $\G$.
A {\it distance partition around} $x$, is a partition $\{\G_0(x),\G(x),\ldots,\G_{\varepsilon(x)}(x)\}$ of $X$.
A {\it $x$-distance-faithful partition} $\{\P_0,\P_1,\ldots,\P_s\}$ with $s\ge\varepsilon$ is a refinement of the distance partition around $x$.
An {\it equitable partition} of a graph $\G$ is a partition $\pi = \{\P_1, \P_2, \dots, \P_s\}$ of its vertex set into nonempty cells such that for all integers $i,j$ $(1 \le i,j \le s)$ the number $c_{ij}$ of neighbours, which a vertex in the cell $\P_i$ has in the cell $\P_j$, is independent of the choice of the vertex in $\P_i$. We call the $c_{ij}$'s the {\it corresponding parameters}. The {\it intersection diagram} of a equitable partition $\pi$ of a graph $\G$ is the collection of circles indexed by the sets of $\pi$ with lines between them. If there is no line between $\P_i$ and $\P_j$, then it means that there is no edge $yz$ for any $y\in\P_i$ and $z\in\P_j$. If there is a line between $\P_i$ and $\P_j$, then a number on the line near circle $\P_i$ denote corresponding parameter $c_{ij}$. A number above or bellow a circle $\P_i$ denote the corresponding parameter $c_{ii}$ (see Figure \ref{2e} for an example).

\begin{figure}[h!]
\begin{center}
\begin{tikzpicture}[scale=.43]
\draw [line width=.8pt] (-7.04,-0.01)-- (0,2.53);
\draw [line width=.8pt] (0,2.53)-- (-0.02,-2.51);
\draw [line width=.8pt] (-0.02,-2.51)-- (6.98,-0.99);
\draw [line width=.8pt] (6.98,-0.99)-- (6.98,1.01);
\draw [line width=.8pt] (6.98,1.01)-- (1,-5.51);
\draw [line width=.8pt] (1,-5.51)-- (1.02,5.51);
\draw [line width=.8pt] (1.02,5.51)-- (-7.04,-0.01);
\draw [line width=.8pt] (1,-5.51)-- (-7.04,-0.01);
\draw [line width=.8pt] (-7.04,-0.01)-- (-0.02,-2.51);
\draw [line width=.8pt] (-0.02,-2.51)-- (6.98,1.01);
\draw [line width=.8pt] (6.98,1.01)-- (1.02,5.51);
\draw [line width=.8pt] (1.02,5.51)-- (0,2.53);
\draw [line width=.8pt] (0,2.53)-- (6.98,-0.99);
\draw [line width=.8pt] (6.98,-0.99)-- (1,-5.51);
\draw [fill=black] (0,2.53) circle [radius=0.2];
\draw [fill=black] (-0.02,-2.51) circle [radius=0.2];
\draw [fill=black] (6.98,-0.99) circle [radius=0.2];
\draw [fill=black] (6.98,1.01) circle [radius=0.2];
\draw [fill=black] (1,-5.51) circle [radius=0.2];
\draw [fill=black] (1.02,5.51) circle [radius=0.2];
\draw [fill=black] (-7.04,-0.01) circle [radius=0.2];
{\tiny
\node at (-7.54,-0.01) {$0$};
\node at (0.5,2.53) {$1$};
\node at (-0.07,-3.01) {$2$};
\node at (7.48,-0.99) {$3$};
\node at (7.48,1.01) {$4$};
\node at (1.5,-5.51) {$5$};
\node at (1.52,5.51) {$6$};
}
\end{tikzpicture}\qquad\qquad
{\small
\begin{tikzpicture}[scale=.47]
\draw (-7.04,-0.01) -- (0,4.01);
\draw (-7.04,-0.01) -- (0.006,-4.014);
\draw (0,4.01) -- (0.006,-4.014);
\draw (0,4.01) -- (7.024,-0.01);
\draw (0.006,-4.014) -- (7.024,-0.01);
\draw [fill=white] (-7.04,-0.01) circle [radius=1.1];
\draw [fill=white] (0,4.01) circle [radius=1.1];
\draw [fill=white] (0.006,-4.014) circle [radius=1.1];
\draw [fill=white] (7.024,-0.01) circle [radius=1.1];
\node at (-7.04,-0.01) {$\P_0$};
\node at (0,4.01) {$\P_1$};
\node at (0.006,-4.014) {$\P_2$};
\node at (7.024,-0.01) {$\P_3$};
\node  at (-6.04,0.99) {$2$};
\node  at (-6.04,-1.01) {$2$};
\node  at (-7.04,-1.51) {--};
\node  at (0,5.51) {$1$};
\node  at (1.3,3.41) {$1$};
\node  at (-1.3,3.41) {$1$};
\node  at (0.3,2.51) {$1$};
\node  at (1.3,-3) {$2$};
\node  at (-1.3,-3) {$1$};
\node  at (0.006,-5.51) {--};
\node  at (0.3,-2.51) {$1$};
\node  at (7.024,-1.51) {$1$};
\node  at (5.624,0.9) {$1$};
\node  at (5.624,-0.7) {$2$};
\end{tikzpicture}
}
\caption{Cayley graph $\mbox{Cay}(\ZZ_{7};\{1,2\})$ and its intersection diagram (around vertex $0$). Adjacency algebra of this graph is closed with respect to Hadamard multiplication (this follows from Theorems \ref{1f} and \ref{1d}; or independently from Theorem \ref{1h}).}
\label{2e}
\end{center}
\end{figure}
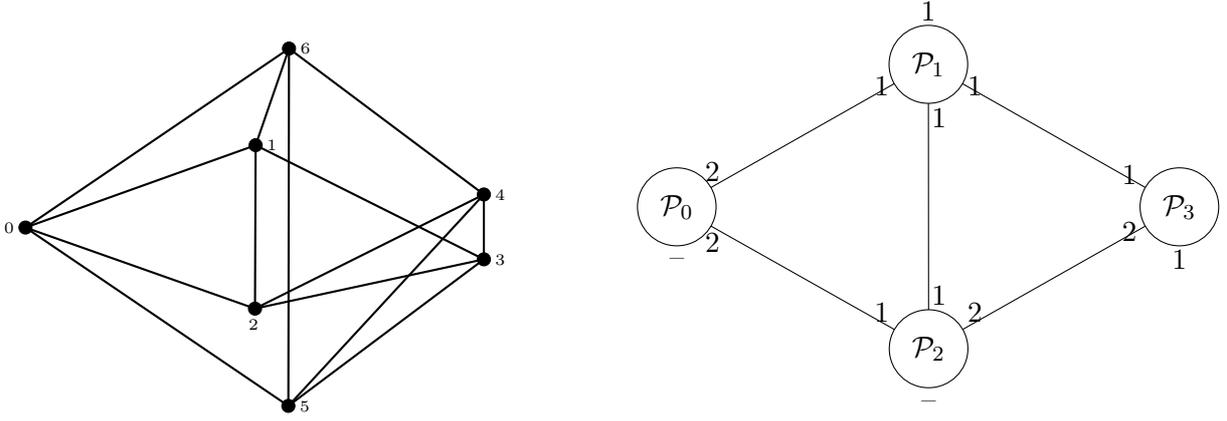


\subsection{The adjacency algebra}
\label{2f}

Let $\CC$ denote the complex number field, and let $\G$ denote a graph with vertex set $X$ and diameter $D$. For $0 \le i \le D$ let $A_i$ denote the matrix in $\Mat_X(\CC)$ with $(x,y)$-entry
\begin{equation}
\label{2g}
  (A_i)_{x y} = \left\{ \begin{array}{ll}
                 1 & \hbox{if } \; \dist(x,y)=i, \\
                 0 & \hbox{if } \; \dist(x,y) \ne i \end{array} \right. \qquad (x,y \in X)
\end{equation}
We call $A_i$ the {\it distance-$i$ matrix} of $\G$. We abbreviate $A:=A_1$ and call this the {\it adjacency matrix} of $\G$. Observe that $A_0 = I$, $\sum_{i=0}^D A_i = J$, $\overline{A_i} = A_i \;(0 \le i \le D)$, and $A_i^\top = A_i  \;(0 \le i \le D)$, where $I$  denotes the identity matrix (respectively, all-ones matrix) in  $\Mat_X(\CC)$.

Let $\V=\CC^X$ denote the vector space over $\CC$ consisting of column vectors whose coordinates are indexed by $X$ and whose entries are in $\CC$. We call $\V$ the {\it standard module}. We endow $\V$ with the Hermitian inner product $\langle \cdot , \cdot \rangle_{\V}$ that satisfies $\langle u,v \rangle_{\V} = u^\top\overline{v}$ for $u,v \in V$, where ``$\top$" denotes transpose and ``$\overline{\phantom{v}}$" denotes complex conjugation.
Recall that
$$
\langle u, Bv \rangle_{\V} = \langle \overline{B}^\top u, v \rangle_{\V}
$$
for $u,v\in \V$ and $B \in \Mat_X(\CC)$.

We observe that $\Mat_X(\CC)$ acts on $\V$ by left multiplication, and since $A$ is a real symmetric matrix, $A$ can be interpret as a self-adjoint operator on $\V$. This yield that $\V$ has an orthogonal basis consisting of eigenvectors of $A$ (see, for example, \cite[Chapter 7]{AS}). Assume that $\G$ has $d+1$ distinct eigenvectors. For each eigenvalue $\lambda_i$ $(0\le i\le d)$ of $\G$ let $U_i$ be the matrix whose columns form an orthonormal basis of its eigenspace $\V_i := \ker(A-\lambda_iI)$, and let $m_i:=\dim(\V_i)$. The {\it primitive idempotents} of $A$ are the matrices
$$
E_i := U_iU_i^\top\qquad(0\le i\le d).
$$
Some well-known properties of the primitive idempotents are the following:
\begin{enumerate}[label=\rm(e-\roman*)]
\item
$p(A)=\sum\limits_{i=0}^d p(\lambda_i) E_i$, for every polynomial $p\in\CC[t]$.
\item
$\trace(E_i)=m_i$ $(0\le i\le d)$.
\item
$E_i^\top=E_i$ $(0\le i \le d)$.
\item
$\G$ regular and connected $\Rightarrow$ $E_0=|X|^{-1}J$.
\item
$E_iE_j=\delta_{ij} E_i$ $(0\le i,j\le D)$.
\item
$E_iA=AE_i=\lambda_i E_i$ $(0\le i\le d)$.
\item
$E_0+E_1+\cdots+E_d=I$.
\item
$\ds{E_i=\frac{1}{\pi_i}\prod_{\stackrel{j=0}{j\not=i}}^d(A-\lambda_jI)}$ $(0\le i\le d)$, where $\pi_i=\prod_{j=0(j\not=i)}^d(\lambda_i-\lambda_j)$.
\item
$A^h=\sum_{i=0}^d \lambda_i^h E_i$ $(h\in\NN)$.
\item{\label{2j}}
$E_i$ is the orthogonal projector onto $\V_i=\ker(A-\lambda_iI)$ $(0\le i \le d)$. Moreover, $\im(E_i)=\ker(A-\lambda_iI)$ and $\ker(E_i)=\im(A-\lambda_iI)$.
\end{enumerate}
Proofs of properties (e-i)--(e-x) can be found, for example, in \cite[Chapter 2]{SP}. Recall that, the number of walks of length $\ell\ge 0$ between vertices $u$ and $v$ of $\G$ is the $(u,v)$-entry of $A^\ell$, and that the eigenvalues of a real symmetric matrix are real numbers (see, for example, \cite{PT}). From this fact, together with (e-iv) and (e-viii), we have the following result:

\begin{corollary}[{Hoffman polynomial}, {\rm{\cite[Theorem 1]{AJH}}}]
\label{2h}
A graph $\G$  is regular and connected if and only if there exists a polynomial $H\in\RR[t]$ such that $J=H(A)$.
\end{corollary}

Now, using the above notation, the vector space
$$
\A=\RR_d[A]=\span\{I,A,A^2,\ldots,A^d\}
$$
is an algebra, with the ordinary product of matrices and orthogonal basis $\{E_0,E_1,\ldots,E_d\}$, called the {\it adjacency algebra}. Moreover, the vector space
$$
\D=\span\{I,A,A_2,\ldots,A_D\}
$$
forms an algebra with the  Hadamard product `$\circ$' of matrices, defined by $(M\circ N)_{uv}=(M)_{uv}(N)_{uv}$. We call $\D$ the {\em distance $\circ$-algebra}. Note that, when $\G$ is regular, $I,A,J\in \A\cap\D$, and thus $\dim(\A\cap\D)\ge 3$ assuming that $\G$ is not a complete graph (in this exceptional case, $J=I+A$). In this algebraic context, an important result is that $\G$ is distance-regular if and only if $\A = \D$, which is therefore equivalent to $\dim(\A\cap \D) = d + 1$ (and hence $d = D$); see, for example, \cite{NB, BCN, RP}. A related concept  was introduced by Weichsel \cite{PW2}: a graph
is called {\em distance-polynomia} if $\D \subset \A$, that is, if each distance matrix is a polynomial in $A$. In other words, a
graph with diameter $D$  is distance-polynomial if and only if $\dim(\A \cap \D) = D + 1$.

In general the algebras $\A$ and $\D$ are different from the algebra $\N=(\langle A_0,A_1,\ldots,$ $A_D \rangle,+,\cdot)$ generated by the set of distance-$i$ matrices $\{ A_0,A_1,\ldots, A_D\}$ with respect to the ordinary product of matrices. Figure \ref{2i} shows a diagram with some inclusion relationships when $\A$ is closed under Hadamard multiplication.

\begin{figure}[t]
\centering
\begin{tikzpicture}
\node (30) at (0,-4.5) {$(\langle I,A,\ldots,A^d\rangle,+,\circ)$};
\node (41) at (3,-6)  {$(\D,+,\circ)$};
\node (42) at (-3,-6)  {$(\A,+,\cdot)$};
\node (50) at (0,-7.5) {$\{I,A,J\}$};
\draw (30) -- (41);
\draw (30) -- (42);
\draw (41) -- (50);
\draw (42) -- (50);
\end{tikzpicture}
\caption{Inclusion diagram when the adjacency algebra $\A$ is closed under Hadamard multiplication. A line segments that goes upward from $M$ to $N$ means that $N$ contains $M$. In case when $\G$ is distance-regular graph we have $\A=\D$.}
\label{2i}
\end{figure}
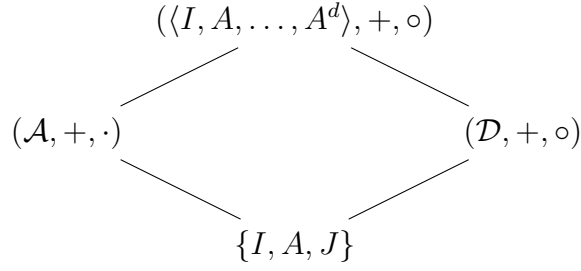


\section{The symmetric association scheme}
\label{3a}
In this section we prove Theorem \ref{1c}.
(If the adjacency algebra $\A=\langle I,A,\ldots,A^d\rangle$ of a graph is closed under Hadamard multiplication, then there exists a basis $\{F_0,F_1,\ldots,F_d\}$ of mutually disjoint $(0,1)$-matrices satisfying the conditions of a symmetric association scheme.)

Let us call two $(0,1)$-matrices $B$, $C$ disjoint if $B\circ C=0$. For the moment, let $\F$ denote a vector space of symmetric $n\times n$ matrices. In \cite[Theorem 2.6.1(i)]{BCN} it was proved that $\F$ has a basis of mutually disjoint $(0,1)$-matrices if and only if $\F$ is closed under Hadamard multiplication. In \cite[Theorem 2.6.1(iii)]{BCN} it was proved that $\F$ is the Bose-Mesner algebra of an association scheme if and only if $I,J\in\F$ and $\F$ is closed under both ordinary and Hadamard multiplication. Thus, in some sense, our Theorem \ref{1c} is a re-proof of \cite[Theorem 2.6.1]{BCN} using a different technique. We emphasize that the notation and technique that we use it the proof of Theorem \ref{1c} is important for the application in Subsection \ref{3c}, as well as for the rest of the paper.

\medskip
\noindent
{\bf Proof of Theorem \ref{1c}.}
Let $X$ denote the vertex set of $\G$ and let $b_i$ $(0\le i\le d)$ denote the row vectors, obtained from $A^i$ $(0\le i\le d)$ as concatenation of the rows of $A^i$. That is, if
$$
A^i=
\left(
\begin{matrix}
d^i_{11} & d^i_{12} & \ldots & d^i_{1,|X|}\\
d^i_{21} & d^i_{22} & \ldots & d^i_{2,|X|}\\
\vdots & \vdots & ~ & \vdots\\
d^i_{|X|,1} & d^i_{|X|,2} & \ldots & d^i_{|X|,|X|}\\
\end{matrix}
\right)
$$
then
$$
b_i=\left(
\begin{matrix}
d^i_{11} & d^i_{12} & \ldots & d^i_{1,|X|} &
d^i_{21} & \ldots & d^i_{|X|,1} & d^i_{|X|,2} & \ldots & d^i_{|X|,|X|}
\end{matrix}
\right).
$$
Define $B$ as the $d\times|X|^2$ matrix constructed from the row set $\{b_0,b_1,\ldots,b_d\}$,
$$
B=\left(
\begin{matrix}
- & b_{0} & - \\
- & b_{1} & - \\
~ & \vdots & ~ \\
- & b_{d} & - \\
\end{matrix}
\right).
$$
It is not hard to see that the vector space $\A$ is isomorphic to the vector space $\C=\C(\Gamma):=\im(B^\top)$,
$$
\C:=\im(B^\top)=\{\gamma_0b_0^\top+\gamma_1b_1^\top+\ldots+\gamma_db_d^\top\,|\,\gamma_0,\gamma_1,\ldots,\gamma_d\in\RR\}.
$$
Using elementary row operation on $B$, we compute $C$ as the reduced row echelon form of the matrix $B$. That is,
$$
B\stackrel{\rm row}{\sim}C=
\left(
\begin{matrix}
1 & * & 0 & 0 & * & * & 0 & * & * & \ldots \\
0 & 0 & 1 & 0 & * & * & 0 & * & * & \ldots \\
0 & 0 & 0 & 1 & * & * & 0 & * & * & \ldots \\
\vdots & ~ & ~ & ~ & \vdots & ~ & 0 & * & ~ & \vdots \\
0 & 0 & 0 & 0 & 0 & 0 & 1 & * & * & \ldots \\
\end{matrix}
\right)=
\left(
\begin{matrix}
- & c_{0} & - \\
- & c_{1} & - \\
~ & \vdots & ~ \\
- & c_{d} & - \\
\end{matrix}
\right).
$$
Note that the set of nonzero vectors $c_i$ $(0\le i\le d)$ are linearly independent. Finally, we can use row vectors $\{c_i\}_{i=0}^d$ to construct our matrices $F_i$ in the following way. If
$$
c_i=
\left(
\begin{matrix}
c^i_{11} & c^i_{12} & \ldots & c^i_{1,|X|} &
c^i_{21} & \ldots & c^i_{|X|,1} & c^i_{|X|,2} & \ldots & c^i_{|X|,|X|}
\end{matrix}
\right).
$$
then
$$
F_i=
\left(
\begin{matrix}
c^i_{11} & c^i_{12} & \ldots & c^i_{1,|X|}\\
c^i_{21} & c^i_{22} & \ldots & c^i_{2,|X|}\\
\vdots & \vdots & ~ & \vdots\\
c^i_{|X|,1} & c^i_{|X|,2} & \ldots & c^i_{|X|,|X|}\\
\end{matrix}
\right).
$$
We claim that the set $\{F_0,F_1,\ldots,F_d\}$ have the required properties. By construction, it is routine to show that the  matrices $F_0,F_1,\ldots,F_m$ are linearly independent.

\medskip
(i) Pick $F_i$ for some $i$ $(0\le i\le d)$. Since $\{F_0,F_1,\ldots,F_d\}$ is a basis of the vector space $\A$, which is closed under both ordinary multiplication and Hadamard multiplication,  there exists scalars $\alpha_0,\ldots,\alpha_d$ such that $F_i\circ F_i=\sum_{h=0}^d \alpha_h F_h$. Now pick $F_j$ (where $j\ne i$) and consider $(x,y)$-entry of $F_j$ which correspond to first nonzero entry of row vector $c_j$. We have $(F_j)_{xy}=1$ and $(F_h)_{xy}=0$ $(0\le h\le d,~h\ne j)$. This yield that if $\alpha_j \ne 0$ then $(F_i\circ F_i)_{xy}=\alpha_j\ne 0$, a contradiction (because $(F_i)_{xy}=0$). Thus $F_i\circ F_i=\alpha_iF_i$. To show that $\alpha_i=1$, pick $(u,v)$-entry of $F_i$ which correspond to first nonzero entry of row vector $c_i$. We have $(F_i)_{uv}=1$ and with that $1=(F_i\circ F_i)_{uv}=(\alpha_iF_i)_{uv}=\alpha_i$.

This yield $F_i\circ F_i= F_i$, and with that all entries of $F_i$ $(0\le i\le d)$ are zeros and ones. On a similar way as above, we can show that $F_i\circ F_j=\boldsymbol{O}$, for $i\ne j$. The result follows.

\medskip
(ii) Since $I\in\A=\span\{F_0,F_1,\ldots,F_d\}$ and the set $\{F_0,F_1,\ldots,F_d\}$ is a basis of $\circ$-idempotents there exists index set $\Omega$ such that $\sum_{\alpha\in\Omega} F_{\alpha}=I$. If $|\Omega|>1$ then we can pick $\alpha\in\Omega$, $y,z\in X$, such that $(I_\alpha)_{yy}=1$ and $(I_\alpha)_{zz}=0$. For an algebra $\A$ we have that for any $B,C\in\A$, $BC=CB$, and since $J\in\A$ we have $I_\alpha J=JI_\alpha$. If we compute $(y,z)$-entry of $I_\alpha J$ and $JI_\alpha$ we get $(I_\alpha J)_{yz}=1$, $(JI_\alpha)_{yz}=0$, a contradiction. The result follows.

\medskip
(iii) Since $\G$ is a regular connected graph we have $J\in\A$. On the other hand, by (i) the set $\{F_0,F_1,\ldots,F_d\}$ is a basis of $\circ$-idempotents. The result follows.

\medskip
(iv) Since $F_i$ $(0\le i\le d)$ are real symmetric matrices, the result follows.

\medskip
(v) Note that $\{F_0,F_1,\ldots,F_d\}$ is a basis of $\A$.
\hfill$\hbox{\rule{3pt}{6pt}}$
\medskip



\subsection{The number of different eigenvalues of a Hermitian matrix}
\label{3c}

As shown in this subsection, the technique used in the proof of Theorem \ref{1c} can be used to find the number of distinct eigenvalues of a symmetric (or Hermitian) matrix.
The motivation for this algorithm is that, for the solution of some problems that deal with eigenvalues,  we only need  to know  the number of different ones.
Moreover, if we have a large matrix (or a set of large matrices), computing all the eigenvalues is time consuming.

Also there is problem of distinct two different eigenvalues when we work with computer programs. All computers work only with rational numbers. So,  if we deal with a large number of eigenvalues, even  when we compute them in the usual way, we always have the problem of distinguishing two of them, because their values are often close to each other, up to some decimal place. Our method can avoid this.

Our method is especially applicable in algebraic and spectral graph theory, since symmetric $(0,1)$-matrix represent adjacency matrix of a graph. Also, for example, see Corollary \ref{gs}. For more information about algebraic and spectral graph theory we recommend \cite{NB,PT}.

As before, let $X$ denote a set with $|X|=n$ elements, $\Mat_X(\CC)$  the set of $n\times n$ matrices over $\CC$ with rows and columns indexed by $X$, and  $A\in\Mat_X(\CC)$ a Hermitian matrix. In this subsection we describe a simple algorithm to find the number $d+1$ of distinct eigenvalues of $A$ (without computing them).
Notice that it suffices to find the dimension of the vector space $\A$ spanned by the powers of $A$.
With this aim, we can
consider the set $\{A^0,A^1,\ldots,A^k\}$ for some positive integer $k$. Then, as in the proof of Theorem \ref{1c}, we construct the matrix $B$ and compute $C=(c_{ij})$ as its reduced row echelon form.
Then, notice that the set of nonzero row vectors $c_i$ $(0\le i\le k)$ are linearly independent. Thus,
we only need to find the smallest $k$ so that $c_k\neq 0$  to conclude that $A$ has $d+1=s+1$ different eigenvalues. The problem with this approach is that to decide what initial number $k$ to pick. Of course, $k=n$ will always work, but, in this case, we need to compute all $A^i$ $(0\le i\le n)$ which is not the best choice if the number of distinct eigenvalues is small compared with $n$.

To overcome the above problem, we propose an algorithm based on the Gram-Schmidt method.
Recall that this method produces a set of orthogonal (and, hence, linearly independent) vectors in an inner product space, which, in our case $\CC_n=\Mat_X(\CC)$ is equipped with the scalar product
\begin{equation}
\label{inner-prod}
\langle A,B\rangle_{\CC_n} :=\frac{1}{n} \trace (A B)=\frac{1}{n}\summ (A\circ \overline{B}), \qquad A,B\in Mat_X(\CC),
\end{equation}
where $\summ(M)$ denotes the sum of all entries of $M$ (the term $\frac{1}{n}$ is a normalization factor to get $\|I\|_{\CC_n}=1$) .
Then, if we apply the method from the matrices $I,A,A^2,\ldots$, we get a sequence $A_0,A_1,\ldots$, where $A_i$ is a polynomial of degree $i$ in $A$, for $i=0,\ldots,d$, $A_0,\ldots,A_{d}$ are orthogonal, and $A_i=0$ for $i>d$. Consequently, we only need to apply the process until we reach the first zero matrix.
Moreover, notice that if, when computing $A_{k+1}$, instead of the power $A^{k+1}$, we  use $A_{k}A$, we have
$\langle A_{k}A, A_i \rangle=\langle A_{k}, AA_i \rangle=0$ for each $i<k-1$ (since $AA_i$ is a polynomial in $A$ of degree $<k$). Then, the algorithm is as follows:

\medskip

\begin{algorithm}
\label{3d}
{\rm
Let $A\in\CC_n=\Mat_X(\FF)$ denote a Hermitian matrix. The following algorithm produces as a result the number of different eigenvalues of matrix $A$, without computing them explicitly.

\medskip
\noindent
{\bf Input:} A Hermitian matrix $A$.\\
\noindent
{\bf Output:} The number $d+1$ of distinct eigenvalues of $A$.
\begin{itemize}
\item[{\bf 1.}]
{\bf Initialize} $A_0:=I$, $A_1:=A$, and $k:=1$.
\item[{\bf 2.}]
{\bf Compute} the matrix $A_{k+1}=A_kA-\sum_{i=k-1}^k \frac{\langle A_kA,A_i \rangle_{\CC_n}}{\|A_i\|_{\CC_n}^2}A_i$.
\item[{\bf 3.}]
{\bf While} $A_{k+1}\neq 0$, {\bf do} $k:=k+1$ and {\bf go to} step {\bf 2}.
\item[{\bf 4.}]
As a result of applying steps {\bf 2} and {\bf 3}, we get a set of matrices $\{A_0,A_1,\ldots,A_r,A_{r+1}\}$ with $A_i\neq 0$ for $i=0,\ldots,r$, and $A_{r+1}=0$
\item[{\bf 5.}]
{\bf Conclude} that $A$ has $d+1=r+1$ different eigenvalues.
\end{itemize}
}
\end{algorithm}

Thus, if $A$ is a Hermitian matrix such that $\A=\span\{A^0,A,\ldots,A^d\}$ is closed under Hadamard product,
we can use Algorithm \ref{3d} to compute the standard basis $\{F_0,F_1,\ldots, F_d\}$ of $\A$ by following the proof of Theorem \ref{1c}(i). Just apply the algorithm to get a set $\{A_0,A_1,\ldots,A_d\}$ of non-zero matrices such that $d+1$ is the number of distinct eigenvalues of $A$, and, starting from them, proceed as in the proof. 

\begin{remark}
\label{r1}
{\rm For application purposes, assume that the entries of $A$ are integers, and we want to work only with integers in the whole procedure avoiding numerical computations. Then, instead of the scalar product in \eqref{inner-prod}, we can use the inner product $\langle A,B\rangle_{\CC_n} =\trace (A B)$ and change Algorithm \ref{3d} accordingly (modifying step 2 also, to avoid devision).}
\end{remark}

\subsection{Checking distance-regularity}
\label{3c2}
If fact, if $A$ is the adjacency matrix of a graph $\G$ with $d+1$ eigenvalues, the above inner product \eqref{inner-prod} is denoted as $\langle \cdot,\cdot\rangle_{\G}$, and the obtained matrices $A_0,A_1,\ldots,A_d$ coincide, up to a multiplicative constant, with the so-called {\em predistance} matrices of $\G$,  see \cite{FP}. In turn, such matrices are obtained by evaluating at $A$ the predistance polynomials $p_0,\ldots,p_d$, introduced in \cite{fg97}.
In particular, if $\G$ is distance-regular, the predistance polynomials and predistance matrices are, respectively, the distance polynomials and distance matrices of $\G$.
If $\G$ has spectrum $\spec(\G) = \spec (A) = \{\lambda_0^{m_0},\lambda_1^{m_1},\dots,
\lambda_d^{m_d}\}$,
where $\lambda_0>\lambda_1>\cdots >\lambda_d$,
the {\em predistance polynomials}
$p_0,p_1,\ldots,p_d$
constitute an orthogonal sequence of polynomials ($\dgr (p_i)=i$) with respect to the scalar product

\begin{equation}\label{inner-prod2}
\langle f, g\rangle_{\G} := \frac{1}{n} \sum_{i=0}^d m_i f(\lambda_i) g(\lambda_i)= \frac{1}{n}\trace (f(A)g(A)) =  \langle f(A), g(A)\rangle_{\G},
\end{equation}
normalized in such a way that
$\|p_i\|_{\G}^2=p_i(\lambda_0)$ (we know that $p_i(\lambda_0)>0$ for every $i=0,\ldots,d$).

As every sequence of orthogonal polynomials, the predistance polynomials satisfy a three-term recurrence of the form
\begin{equation}
\label{recur}
xp_{i}=b_{i-1}p_{i-1}+a_ip_i+c_{i+1}p_{i+1}\qquad (0\le i\le d),
\end{equation}
where the constants $b_{i-1}$, $a_i$, and $c_{i+1}$ are the Fourier coefficients of $xp_i$ in terms of $p_{i-1}$, $p_i$, and $p_{i+1}$, respectively (and $b_{-1}=c_{d+1}=0$).
Moreover, $p_0+p_1+\cdots+p_d=H$, the Hoffman polynomial of Corollary \ref{2h}. Hence, if $\G$ is $k$-regular, we can apply Algorithm \ref{3d} to obtain the predistance matrices if we normalize each $A_i$, for $i=0,\ldots,d$, in such a way that $\|A_i\|_{\G}^2=\langle A_i,J \rangle_{\G}$, which satisfy
\begin{equation}
\label{sum-p=J}
A_0+A_1+\cdots + A_d = p_0(A)+p_1(A)+\cdots +p_d(A)=H(A)=J.
\end{equation}

Some recent characterizations of distance-regularity in terms of the predistance polynomials and distance matrices $A_d$ and $A_{d-1}$ are the following:
A regular graph $\G$ with $d + 1$ distinct eigenvalues, diameter $D = d$, is distance-regular if and only if either
\begin{enumerate}
[label=\rm(DR\arabic*)]
\item
$A_d\in\A$,
\item
$A_d = p_d(A)$,
\item
$A_i = p_i(A)$ for $i=d-2,d-1$.
\end{enumerate}
Every of the above conditions assures the existence of all the distance matrices $A_0(=I), A_1(=A),A_2,\ldots,A_d$, which is a well-known characterization of distance-regularity. More generally, in \cite{DDF}, a graph $\G$ is said to be {\em $k$-partially distance-regular},
for some $k<d$, if there exist the distance matrices $A_i$ for $i=0,\ldots,k$.
For more details, see \cite{EvD, DDF, Fpdr, FGY}.

Now, as another possible application of Algorithm  \ref{3d} we have the following result.
\begin{proposition}
\label{propo1}
Let $\G$ be a regular graph with diameter $D$, and $d+1$ different eigenvalues. Let $A_i$ be the matrices obtained by applying the Algorithm \ref{3d}, and normalizing them so that $\|A_i\|_{\G}^2=\langle A_i,J \rangle_{\G}$, for $i=0,1\ldots$, that is,
$A_i \leftarrow \frac{\langle A_i, J\rangle_{\G}}{\|A_i\|_{\G}^2}A_i$. If the following conditions hold:
\begin{enumerate}[label=\rm(\roman*)]
\item
$A_{D+1}=0$ and $A_D\neq 0$,
\item
$A_D$ is a $(0,1)$-matrix,
\item
$A_i$, $i=0,\ldots,D-1$, are nonnegative matrices.
\end{enumerate}
then, $\G$ is a distance-regular graph.
\end{proposition}

\begin{proof}
We will prove that $A_d$ is the $d$-distance matrix of $\G$.
First, as we have already seen, (i) implies that $D=d$.
Then, if $u,v\in X$ are two vertices at distance $\dist(u,v)=d$, we have that
$(A_d)_{uv}=(p_d(A))_{uv}=(H(A))_{uv}=(J)_{uv}=1$.
Otherwise, assume that $\dist(u,v)=\ell<d$ and $(A_d)_{uv}=1$. Then, from \eqref{sum-p=J}
and (iii), it should be $(A_{\ell}+\cdots+A_{d-1})_{uv}=0$. In particular,
$(A_{\ell})_{uv}=0$, a contradiction since $A_{\ell}=p_{\ell}(A)$, with $\dgr(p_{\ell})=\ell$ and so $p_{\ell}$ has leading nonzero coefficient.
Then, if $\dist(u,v)<d$, then $(A_d)_{uv}=0$. Consequently, $A_d$ is as claimed, and (DS2) gives the result.
\end{proof}

Notice that, in fact, if $\G$ is indeed distance-regular, all the normalized matrices $A_0,A_1,\ldots$ obtained by the Algorithm \ref{3d} must be the corresponding distance matrices.

\begin{algorithm}
\label{3d2}
{\rm
The following algorithm returns `true' or `false' depending on whether a regular graph is distance-regular or not.

\medskip
\noindent
{\bf Input:} The adjacency matrix $A$ of a regular graph $\G$.\\
\noindent
{\bf Output:} `true' ($\G$ is distance-regular) or `false'.
\begin{itemize}
\item[{\bf 1.}]
{\bf Initialize} $A_0:=I$, $A_1:=A$, and $k:=1$.
\item[{\bf 2.}]
{\bf Compute} the matrix $A_{k+1}=A_kA-\sum_{i=k-1}^k \frac{\langle A_kA,A_i \rangle_{\G}}{\|A_i\|_{\G}^2}A_i$,
\item[{\bf 3.}]
{\bf Set} $A_{k+1} := \frac{\langle A_{k+1}, J\rangle_{\G}}{\|A_{k+1}\|_{\G}^2}A_{k+1}$
(normalization),
\item[{\bf 4.}]
{\bf If} $A_{k+1}$ is not a $(0,1)$-matrix, {\bf then}  {\bf return} false, and end the program.
\item[{\bf 5.}]
{\bf If} $A_{k+1}=0$ and $k<D$ {\bf then} {\bf return} false, and end the program.
\item[{\bf 6.}]
{\bf If}$A_{k+1}=0$ and $k=D$ {\bf then} {\bf return} true, and end the program.
\item[{\bf 7.}] {\bf Set} $k:=k+1$ and  {\bf go to} step {\bf 2}.
\end{itemize}
}
\end{algorithm}

\begin{remark}{\rm 
Here, a comment similar to  Remark \ref{1c} is in order. Indeed, notice that in Proposition \ref{propo1} and Algorithms \ref{3d} and \ref{3d2} the normalization of the matrices $A_i$ is not strictly neccessary (and in the algorithms is time consuming). We only need to require that all entries of $A_i$ have the same value, say, $c_i$. Then, if eventually  we want to get $(0,1)$-matrices, we simply apply $A_i\leftarrow \frac{1}{c_i}A_i$.}
\end{remark}


\section{The distance-faithful intersection diagrams}
\label{7a}

In this section we prove Theorem \ref{1g}. (If the adjacency algebra $\A=\{I,A,\ldots,A^d\}$ of a regular graph is closed under Hadamard multiplication, then there exist a common $x$-distance-faithful intersection diagram with $d+1$ cells for every vertex $x$.)

\medskip
\noindent
{\bf Proof of Theorem \ref{1g}.}
Since $\G$ is a regular graph, by Theorem \ref{1c} $\A$ has the standard basis $\{F_0,F_1,\ldots,F_d\}$. Let $X$ denote the vertex set of $\G$, pick two vertices $x,u\in X$ and define partitions $\pi_x$ and $\pi_u$ of $X$ in the following way
$$
\pi_x=\{\P_0(x),\P_1(x),\ldots,\P_d(x)\},
\qquad\mbox{ where }\quad
\P_i(x)=\{z\mid  (F_i)_{xz}=1\}~(0\le i\le d),
$$
$$
\pi_u=\{\P_0(u),\P_1(u),\ldots,\P_d(u)\},
\qquad\mbox{ where }\quad
\P_i(u)=\{w \mid (F_i)_{uw}=1\}~(0\le i\le d).
$$
To prove the claim, we need to show that the following (i)--(iii) hold.
\begin{enumerate}[label=\rm(\roman*)]
\item
All vertices in $\P_i(x)$ are on the same distance from $x$.
\item
$|\P_i(x)|=|\P_i(u)|$ $(0\le i\le d)$.
\item
There exist numbers $c_{ij}$ $(0\le i,j\le d)$ such that
\begin{enumerate}[label=\rm(\arabic*)]
\item
$\pi_x$ is equitable partition of $\G$ with corresponding parameters $c_{ij}$.
\item
$\pi_u$ is equitable partition of $\G$ with corresponding parameters $c_{ij}$.
\end{enumerate}
\end{enumerate}

\medskip
(i) We will first show that for any $z,w\in\P_i(x)$ we have $(A^\ell)_{xz}=(A^\ell)_{xw}$ $(0\le \ell\le d)$, that is, the number of walks of length $\ell$ from $x$ to $z$ is the same as the number of walks of length $\ell$ from $x$ to $w$. Since $\{F_h\}_{h=0}^d$ is a basis of $\A$ there exist scalars $\alpha_{ij}$ $(0\le i,j\le d)$ such that
$$
A^\ell=\sum_{j=0}^d \alpha_{\ell j} F_j
\qquad
(0\le \ell\le d).
$$
Since $z,w\in\P_i(x)$ we have $(F_i)_{xz}=(F_i)_{xw}=1$ and $(F_j)_{xz}=(F_j)_{xw}=0$ for $j\ne i$. This yields $(A^\ell)_{xz}=\alpha_{\ell i}=(A^\ell)_{xw}.$

Now we prove the claim (i) by contradiction. Assume that $z,w\in\P_i(x)$ and that $\dist(x,z)>\dist(x,w)=\ell$. Then, we  have $(A^{\ell})_{xw}\ne 0$ but $(A^{\ell})_{xz}= 0$, a contradiction.

\medskip
(ii) Pick $i$ $(0\le i\le d)$. If $\G$ is a regular graph of valency $k$, then $A\jj=k\jj$ (where $\jj$ is all-ones column vector). This yields $E_0\jj=\jj$ and $E_j\jj=\0$ for $1\le j\le d$ (see property (e-x) from page \pageref{2j}). Now, since $F_i\in\A=\span\{E_0,E_1,\ldots,E_d\}$, there exist scalars $\beta_{h}$ $(0\le h\le d)$ such that
$$
F_i=\sum_{h=0}^d \beta_h E_h.
$$
This implies $F_i\jj=\beta_0 E_0\jj = \beta_0 \jj$, that is, the sum of row entries is the same for every vertex. Therefore, $|\P_i(x)|=\sum_{z\in X} (F_i)_{xz}=\beta_0=\sum_{w\in X} (F_i)_{uw}=|\P_i(u)|$.

\medskip
(iii) Since $AF_i\in\span\{F_0,F_1,\ldots,F_d\}$, there exist scalars $c_{ij}$ $(0\le i,j\le d)$ such that
\begin{equation}
\label{7b}
AF_i=\sum_{h=0}^d c_{ih} F_h
\qquad(0\le i\le d).
\end{equation}
Pick $y\in\P_j(x)$. Now, from the left side of \eqref{7b} we have
$$
(AF_i)_{yx}=\sum_{z\in X} (A)_{yz} (F_i)_{zx} = |\G(y)\cap \P_i(x)|,
$$
and from the right side of \eqref{7b} we have
$$
(AF_i)_{yx}=\left(\sum_{h=0}^d c_{ih} F_h\right)_{yx}= c_{ij} (F_j)_{yx} = c_{ij}.
$$
Thus, $\pi_x$ is an equitable partition of $\G$ with corresponding parameters $c_{ij}$.
Similarly, pick $v\in\P_j(u)$. From one side of \eqref{7b} we have
$
(AF_i)_{vu}=|\G(v)\cap \P_i(u)|
$
and from the other side of \eqref{7b},
$
(\sum_{h=0}^d c_{ih} F_h)_{vu}= c_{ij}.
$
Therefore,  $\pi_u$ is also an equitable partition of $\G$ with corresponding parameters $c_{ij}$.
\hfill$\hbox{\rule{3pt}{6pt}}$
\medskip


\section{The quotient-polynomial graphs}
\label{go}

In this section we recall some old, and prove some new, properties of quotient-polynomial graphs. Recall that, for every $y,z\in X$, $(A^\ell)_{yz}$ $(0\le \ell\le d)$ is the number of walks of length $\ell$ between vertices $y$ and $z$.

\begin{definition}{\rm
\label{gb}
Let $\G$ denote a graph with vertex set $X$ and $d+1$ distinct eigenvalues. The column vector $\ww(y,z)\in\CC^{d+1}$ is defined as
$$
\ww(y,z)=\Big( (A^0)_{yz}, (A^1)_{yz},\ldots,(A^d)_{yz} \Big)^{\top}.
$$
}\end{definition}

\begin{remark}\label{gm}{\rm
If we have an equitable partition $\pi=\{\P_0,\P_1,\ldots,\P_r\}$ around $y$, $\P_0=\{y\}$,
with intersection numbers $b_{ij}$ we can compute the vector $\ww(y,z)$ $(y,z\in X)$ from its {\em quotient matrix} $B=(b_{ij})\in\Mat_{(r+1)\times (r+1)}(\CC)$, $(0\le i, j\le r)$.
The reason is that $\frac{1}{|\P_j|}(B^\ell)_{\P_j,\P_0}$ is the number of $\ell$-walks $(0\le\ell\le d)$ from $z$ to $y$ for any $z\in\P_j$ $(0\le j\le r)$ (see, for instance, \cite{DF}).
}\end{remark}

\begin{lemma}
\label{gc}
Let $\R=\{R_0,R_1,\ldots,R_r\}$ denote a partition of $X\times X$ such that, for each $i$ $(0\le i\le r)$, the pairs $(y,z),(u,v)\in X\times X$ belong to $R_i$ if and only if $\ww(y,z)=\ww(u,v)$. Then all pairs of vertices in a given $R_i$ are at the same distance.
\end{lemma}

\begin{proof}
By contradiction, assume that $(y,z),(u,v)\in R_i$ and that $\dist(y,z)>\dist(u,v)=\ell$. Then, we would have $(A^{\ell})_{uv}\ne 0$ but $(A^{\ell})_{yz}= 0$, against the definition of $R_i$.
\end{proof}

\begin{definition}
\label{ga}
{\rm
A partition $\R=\{R_0,R_1,\ldots,R_r\}$ of $X\times X$ is called {\it walk-regular} if, for each $i$ $(0\le i\le r)$, the pairs $(y,z),(u,v)\in X\times X$ belong to $R_i$ if and only if $\ww(y,z)=\ww(u,v)$. Let $M_i$ $(0\le i\le r)$ denote the $|X|\times|X|$ matrix, indexed by the vertices of $\G$, and defined by
$$
(M_i)_{yz}=\left\{ \begin{array}{ll}
                 1 & \hbox{if } \; (y,z)\in R_i\,\\
                 0 & \hbox{otherwise. } \end{array} \right. \qquad (y,z \in X).
$$
The matrix $M_i$ is called {\it adjacency matrix of the equivalence class $R_i$}.
}\end{definition}

\begin{remark}
\label{gv}
{\rm
Note that we always can permute indices of $\{R_0,R_1,\ldots,R_r\}$ of a walk-regular partition. So, if necessary and using Lemma \ref{gc}, we can define a walk-regular partition by adding the following restriction on $\R$: for any $i\le j$ and $(x,y)\in R_i$, $(u,v)\in R_j$ we have $\dist(x,y)\le\dist(u,v)$.
}
\end{remark}

\begin{lemma}
\label{gd}
Let $\G$ be a graph with vertex set $X$ and a walk-regular partition $\R$ of $X\times X$.
Let $A_i$ $(0\le i\le D)$ denote the distance-$i$ matrix of $\G$, and let $M_i$ $(0\le i\le r)$ denote the adjacency matrices of the corresponding equivalence classes $R_i$. Then there exists an index set $\Phi_i\subset \{0,\ldots,r\}$ such that
$$
A_i=\sum_{j\in\Phi_i} M_j.
$$
\end{lemma}

\begin{proof}
Immediate from Lemma \ref{gc}.
\end{proof}

\begin{definition}
\label{gk}{\rm
Let $\G$ denote a graph with vertex set $X$, $d+1$ distinct eigenvalues, and adjacency algebra $\A$. Let $\R=\{R_0,R_1,\ldots,R_r\}$ denote the walk-regular partition of $X\times X$ and let $M_i$ $(0\le i\le r)$ denote the adjacency matrices of the equivalence classes $R_i$ $(0\le i\le r)$. A graph $\G$ is quotient-polynomial if $M_i\in\A$ $(0\le i\le r)$.
}\end{definition}

From this definition and Lemma \ref{gd} it follows that every distance-$i$ matrix of $\G$ belongs to its adjacency algebra $\A$.

\begin{example}
\label{gp}
{\rm
Let $B\otimes C$ denote the Kronecker tensor product of matrices $B$ and $C$ (for the definition and properties of Kronecker tensor product see, for example, \cite[Chapter 13]{AL} or \cite[Chapter 4]{HF}). Let $A$ and $A'$ denote the adjacency matrices of the graphs $\G$ and $\G'$ respectively. The Kronecker product, $\G\otimes\G'$, is that graph with adjacency matrix $A\otimes A'$ (see \cite{PW}).

Let $T_4$ be the triangular graph with  vertex set $X'=\{0,1,2,3,4,5\}$, and edge set $R'=\{01,02,03,12,13,14,24,25,34,35,45\}$ (that is, $T_4$ is the line graph of the complete graph $K_4$). The distinct eigenvalues of $T_4$ are $\{-2,0,4\}$, and the distinct eigenvalues of the complete graph $K_2$ are $\{-1,1\}$. Consider the graph $\G=K_2\otimes T_4$.
From \cite[Theorem~13.12]{AL}, the distinct eigenvalues of $\G$ are $\{-4,-2,0,2,4\}$, and from \cite[Theorem~1]{PW}, $\G$ is connected.
Moreover, $\G$ is a quotient-polynomial graph. The adjacency algebra of $\G$ is closed with respect to the Hadamard product, and has the standard basis $\{F_0,F_1,F_2,F_3,F_4\}$, where $F_i:=p_i(A)$ $(0\le i\le 4)$ and
$$
p_0(t)=1,\qquad
p_1(t)=t,\qquad
p_2(t)=-\frac{t^4}{32}+\frac{5t^2}{8}-1,
$$
$$
p_3(t)=\frac{t^4}{16}-\frac{3 t^2}{4},\qquad
p_4(t)=\frac{t^3}{8}-\frac{3t}{2}.
$$
For the corresponding intersection diagram of $\G$ see Figure \ref{gq}.
}\end{example}

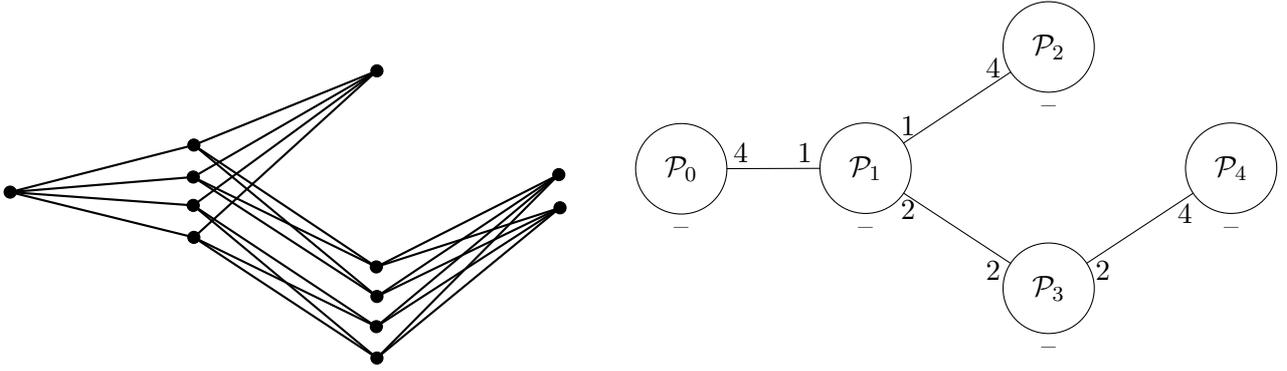
\begin{figure}
{\small
\begin{center}
\begin{tikzpicture}[scale=.4]
\draw [line width=.8pt] (5.98,-2.51)-- (12.02,-0.55);
\draw [line width=.8pt] (5.98,-2.51)-- (-0.02,1.53);
\draw [line width=.8pt] (5.98,-2.51)-- (-0.04,0.47);
\draw [line width=.8pt] (5.98,-2.51)-- (11.98,0.55);
\draw [line width=.8pt] (6,3.99)-- (-0.04,-0.47);
\draw [line width=.8pt] (6,3.99)-- (-0.02,1.53);
\draw [line width=.8pt] (6,3.99)-- (-0.04,0.47);
\draw [line width=.8pt] (6,3.99)-- (-0.02,-1.53);
\draw [line width=.8pt] (6,-5.53)-- (-0.04,-0.47);
\draw [line width=.8pt] (6,-5.53)-- (12.02,-0.55);
\draw [line width=.8pt] (6,-5.53)-- (-0.02,-1.53);
\draw [line width=.8pt] (6,-5.53)-- (11.98,0.55);
\draw [line width=.8pt] (5.98,-4.49)-- (-0.04,-0.47);
\draw [line width=.8pt] (5.98,-4.49)-- (12.02,-0.55);
\draw [line width=.8pt] (5.98,-4.49)-- (-0.02,-1.53);
\draw [line width=.8pt] (5.98,-4.49)-- (11.98,0.55);
\draw [line width=.8pt] (6,-3.49)-- (12.02,-0.55);
\draw [line width=.8pt] (6,-3.49)-- (-0.02,1.53);
\draw [line width=.8pt] (6,-3.49)-- (-0.04,0.47);
\draw [line width=.8pt] (6,-3.49)-- (11.98,0.55);
\draw [line width=.8pt] (-6.06,-0.03)-- (-0.04,-0.47);
\draw [line width=.8pt] (-6.06,-0.03)-- (-0.02,1.53);
\draw [line width=.8pt] (-6.06,-0.03)-- (-0.04,0.47);
\draw [line width=.8pt] (-6.06,-0.03)-- (-0.02,-1.53);
\draw [fill=black] (5.98,-2.51) circle [radius=0.2];
\draw [fill=black] (6,3.99) circle [radius=0.2];
\draw [fill=black] (6,-5.53) circle [radius=0.2];
\draw [fill=black] (5.98,-4.49) circle [radius=0.2];
\draw [fill=black] (6,-3.49) circle [radius=0.2];
\draw [fill=black] (-6.06,-0.03) circle [radius=0.2];
\draw [fill=black] (-0.04,-0.47) circle [radius=0.2];
\draw [fill=black] (12.02,-0.55) circle [radius=0.2];
\draw [fill=black] (-0.02,1.53) circle [radius=0.2];
\draw [fill=black] (-0.04,0.47) circle [radius=0.2];
\draw [fill=black] (-0.02,-1.53) circle [radius=0.2];
\draw [fill=black] (11.98,0.55) circle [radius=0.2];
\end{tikzpicture}
\qquad
\begin{tikzpicture}[scale=.4]
\draw (-6.06,-0.03)-- (0,-0.01);
\draw (6.02,4.01)-- (0,-0.01);
\draw (6.02,-3.99)-- (0,-0.01);
\draw (6.02,-3.99)-- (12.02,-0.01);
\draw [fill=white] (-6.06,-0.03) circle (1.5cm);
\draw [fill=white] (0,-0.01) circle (1.5cm);
\draw [fill=white] (6.02,4.01) circle (1.5cm);
\draw [fill=white] (6.02,-3.99) circle (1.5cm);
\draw [fill=white] (12.02,-0.01) circle (1.5cm);
\node at (-6.06,-0.03) {$\P_0$};
\node at (0,-0.01) {$\P_1$};
\node at (6.02,4.01) {$\P_2$};
\node at (6.02,-3.99) {$\P_3$};
\node at (12.02,-0.01) {$\P_4$};
\node at (-6.06,-2.03) {--};
\node at (0,-2.01) {--};
\node at (6.02,2.01) {--};
\node at (6.02,-5.99) {--};
\node at (12.02,-2.01) {--};
\node at (-4.1,0.5) {$4$};
\node at (-2,0.5) {$1$};
\node at (1.4,1.4) {$1$};
\node at (1.4,-1.4) {$2$};
\node at (4.2,3.3) {$4$};
\node at (4.2,-3.4) {$2$};
\node at (7.8,-3.4) {$2$};
\node at (10.5,-1.5) {$4$};
\end{tikzpicture}
\caption{The quotient-polynomial graph $\G:=K_2\otimes T_4$ and its intersection diagram. The adjacency algebra of $\G$ is closed with respect to the Hadamard product. If $\{F_0,F_1,F_2,F_3,F_4\}$ is the standard basis from Remark \ref{gp}, then for a fixed vertex $x$ of $\G$ we have $\P_i=\{z \mid (F_i)_{xz}=1\}$ $(0\le i\le d)$.}
\label{gq}
\end{center}
}\end{figure}

\begin{definition}
\label{gl}{\rm
Let $\G$ denote a graph with $d+1$ distinct eigenvalues. Given a walk-regular partition $\R=\{R_0,R_1,\ldots,R_r\}$ of $X\times X$, let $w_{ij}$ be the common value of the number of $i$-walks $(0\le i\le d)$ from $y$ to $z$ for any $y,z\in R_j$ $(0\le j\le r)$. Define the matrices $W$ and $Z$, and the polynomials $p_i(t)$ $(0\le i\le d)$ as follows:
$$
[W |\tt]=
\renewcommand\arraystretch{1.3}
\mleft[
\begin{array}{cccc|c}
w_{00} & w_{01} & \ldots & w_{0r} & 1\\
w_{10} & w_{11} & \ldots & w_{1r} & t\\
w_{20} & w_{21} & \ldots & w_{2r} & t^2\\
\vdots & \vdots & \, & \vdots & \vdots\\
w_{d0} & w_{d1} & \ldots & w_{dr}  & t^d
\end{array}
\mright]
\stackrel{\rm row}{\sim}
\mleft[
\begin{array}{ccccccccc|c}
1 & * & 0 & 0 & \hdots&0& *&\hdots&*& p_1(t)\\
0 & 0 & 1 & 0 & \hdots&0& *&\hdots&*& p_1(t)\\
0 & 0 & 0 & 1 & \hdots&0& *&\hdots&*& p_2(t)\\
\vdots & \vdots & \vdots & \vdots & ~& \vdots& \vdots & ~ & \vdots & \vdots\\
0 & 0 & 0 & 0 & \hdots&1& *&\hdots&* & p_d(t)
\end{array}
\mright]=
[Z |\pp(t)],
$$
that is, the matrix $[Z|\pp(t)]$ is the reduced row-echelon form of $[W|\tt]$.
}\end{definition}

\begin{theorem}
\label{gf}
Let $\G$ be a graph with vertex set $X$, $d+1$ distinct eigenvalues, and let $\R=\{R_0,R_1,\ldots,R_r\}$ denote a walk-regular partition of $X\times X$. Then,
$$
d\le r.
$$
Furthermore, let $Z$ denote the matrix of Definition \ref{gl}, and define
$
\W:=\left\{\ww(y,z) \mid y,z\in X \right\}$. Then the following are equivalent.
\begin{enumerate}[label=\rm(\roman*)]
\item
$d=r$.
\item
$Z=I$.
\item
$|\W|=d+1$.
\item
$\W$ is a linearly independent set.
\item
$\G$ is a quotient-polynomial graph.
\end{enumerate}
\end{theorem}

\begin{proof}
Let $M_j$ denote the adjacency matrix of the equivalent-class $R_j$ $(0\le j\le r)$. Since $\R$ is a walk-regular partition, for the scalars $w_{ij}$ $(0\le i\le d,~0\le j\le r)$ of Definition \ref{gl}, we have
\begin{eqnarray*}
 I & = & w_{00} M_{0} + w_{01} M_1 + \cdots + w_{0r} M_r,\\
 A & = & w_{10} M_{0} + w_{11} M_1 + \cdots + w_{1r} M_r,\\
 A^2 & = & w_{20} M_{0} + w_{21} M_1 + \cdots + w_{2r} M_r,\\
\,& \, & \vdots\\
A^d & = & w_{d0} M_{0} + w_{d1} M_1 + \cdots + w_{dr} M_r.
\end{eqnarray*}
This yields $\span\{I,A,\ldots,A^d\}\subseteq \span\{M_0,M_1,\ldots,M_r\}$ as vector spaces, and hence $d\le r$.

Let $W$ denote the matrix from Definition \ref{gl}. Note that the elements of the set $\W$ are columns of the matrix $W$, and since $\R$ is a walk-regular partition, $\W$ has exactly $r+1$ elements.

Also note that
\begin{equation}
\label{gu}
\rank(W)\ge d+1.
\end{equation}
Otherwise, if $\rank(W)<d+1$, applying elementary row operations on the above system, we get $A^d\in\span\{I,A,\ldots,A^{d-1}\}$, a contradiction.

To prove equivalences between (i)--(v), we show the following chain of implications.

\medskip
(i) $\Rightarrow$ (ii), (v).
If $d=r$ then $\rank(W)= d+1 = r+1$, which means that $Z=I$ and for every $M_i$ we have $M_i=p_i(A)$. This yields $M_i\in\A$, and $\G$ is a quotient-polynomial graph.

\smallskip
(ii) $\Rightarrow$ (i), (iii), (iv). If $Z=I$, since $Z$ is a $(d+1)\times(r+1)$ matrix, we have $r=d$. Moreover, we also have that $\rank(W)=d+1$. This yield $|\W|=d+1$ and $\W$ is a linearly independent set.

\smallskip
(iii) $\Rightarrow$ (i), (iv).
If $|\W|=d+1$ then $d=r$ (since $\W$ has $r+1$ elements). If $\W$ is a linearly dependent set, then $\rank(W)<d+1$, which is a contradiction with \eqref{gu}.

\smallskip
(iv) $\Rightarrow$ (i).
If $\W$ is a linearly independent set, then $\rank(W)\ge r+1$. On the other hand, since $d\le r$, and $W$ is $(d+1)\times(r+1)$ matrix, we have $\rank(W)\le d+1$. This yield $d=r$.

\smallskip
(v) $\Rightarrow$ (i).
If $\G$ is a quotient-polynomial graph then $M_i\in\A$ $(0\le i\le r)$. Then as vector spaces $\span\{M_0,M_1,\ldots,M_r\}\subseteq\span\{I,A,\ldots,A^d\}$, which yield $r\le d$. On the other hand, since $d\le r$, the result follows.
\end{proof}

\begin{corollary}
\label{gs}
Let $\G$ denote a graph with $d+1$ distinct eigenvalues, and $x$-distance-faithful intersection diagram $\pi$ with $r+1$ cells. If $\G$ has the same $x$-distance-faithful intersection diagram around every vertex $x$, then $\G$ has at most $r+1$ eigenvalues. Moreover, if $r=d$ then $\G$ is a quotient-polynomial graph.
\end{corollary}

\begin{proof}
The same intersection diagram around every vertex corresponds to a walk-regular partition of $X\times X$ with $r+1$ cells. The result now follows from Theorem \ref{gf}.
\end{proof}

Considering the proof of Theorem \ref{gf}, the number of distinct eigenvalues of $A^i$ $(0\le i\le d)$ is important in deciding when $\G$ is not a quotient-polynomial graph.

\begin{corollary}
\label{gs2}
Let $\G$ denote a graph with vertex set $X$ and $d+1$ distinct eigenvalues. If, for $i\in \{0,\ldots,d\}$, the matrix $A^i$ has more than $d+1$ distinct eigenvalues, then $\G$ is not a quotient-polynomial graph.
\end{corollary}

\begin{proof}
Under the hypothesis, $A^i$ cannot be written as a linear combination of some $d+1$ $\circ$-idempotent $(0,1)$-matrices in $\{F_0,\ldots, F_d\}$ and, hence, $\A$  does not have a standard basis.
\end{proof}

\begin{comment}
\label{gr}{\rm
If $\G$ is a quotient-polynomial graph then the polynomials $p_i$ $(0\le i\le r)$ from Definition \ref{gl} are orthogonal with respect to the scalar product \eqref{inner-prod2}, as happens with the distance polynomials of a distance-regular graph.
Indeed, for every $i,j$ $(0\le i,j\le d)$, we have
$$
\langle p_i,p_j\rangle_{\G}= \langle p_i(A),p_j(A) \rangle_{\G}= \langle M_i,M_j \rangle_{\G}
=\frac{1}{|X|}\sum\limits_{u,v\in X}(M_i\circ \ol{M_j})_{uv}
=0.
$$
Also, for the same polynomials $p_i$ $(0\le i\le r)$, we have that $\G$ is a regular and connected graph if and only if $\sum_{i=0}^r p_i(A)=J$.
}\end{comment}

\medskip
\noindent
{\bf Proof of Theorem \ref{1f}.} (If $\G$ has the same $x$-distance-faithful intersection diagram with $r$ cells around every vertex, then $\G$ has exactly $\rank(P)$ distinct eigenvalues, where $P=(w_{ij})_{(r+1)\times (r+1)}$. If $\rank(P)=r+1$ then $\G$ is a quotient-polynomial graph.)

Using the intersection diagram $\pi_x=\{\P_0,\P_1,\ldots,\P_r\}$ around $x$, we can consider the column vectors
\begin{equation}
\label{gt}
\ww_0=\left(\begin{matrix}
w_{00}\\w_{10}\\w_{20}\\ \vdots\\w_{r0}\\
\end{matrix}\right),
\ww_1=\left(\begin{matrix}
w_{01}\\w_{11}\\w_{21}\\ \vdots\\w_{r1}\\
\end{matrix}\right),
\ldots,
\ww_r=\left(\begin{matrix}
w_{0r}\\w_{1r}\\w_{2r}\\ \vdots\\w_{rr}\\
\end{matrix}\right),
\end{equation}
where $w_{ij}$ denote the number of $i$-walks $(0\le i\le r)$ from $z$ to $x$ for any $z\in\P_j$ $(0\le j\le r)$. Note that we do not know is it $\ww_i\ne\ww_j$ for every $0\le i,j\le r$. Now, pick a vertex $u\in X$ $(u\ne x)$, consider the intersection diagram $\pi_u=\{\P_0(u),\P_1(u),\ldots,\P_r(u)\}$, and let $\ww'_{ij}(u,v)$ denote the number of $i$-walks $(0\le i\le r)$ from $v$ to $u$ for any $v\in\P_j(u)$ $(0\le j\le r)$. Then, since $\G$ has the same intersection diagram around every vertex, the set of vectors
$$
\ww'_0(u,v),
\ww'_1(u,v),
\ldots,
\ww'_r(u,v),
$$
is the same  as in \eqref{gt}. That is, for every $i$ $(0\le i\le r)$ there exists exactly one $h$ $(0\le h\le r)$ such that       $\ww_i=\ww'_h(u,v)$. Now we can define the matrices $M_i\in\Mat_{(r+1)\times(r+1)}(\CC)$ in the following way:
$$
(M_i)_{zy} =
\left\{ \begin{array}{ll}
1 & \hbox{if } \; \ww'_h(z,y)=\ww_i \hbox{ for some } h, \\
0 & \hbox{otherwise } \;
\end{array} \right.
\qquad (z,y \in X).
$$
This definition of $M_i$ yields that
\begin{equation}
\label{gw}
A^i=w_{i0}M_0+w_{i1}M_1+\cdots+w_{ir}M_r
\qquad(0\le i\le r).
\end{equation}
Also, since $\G$ has the same distance-faithful intersection diagram around every vertex, using this intersection diagram we can construct a walk-regular partition of $X\times X$ with $r+1$ basis relations $R_i$. So, by Theorem \ref{gf}, $d\le r$. By assumptions
$$
P=
\mleft[
\begin{array}{cccc}
w_{00} & w_{01} & \ldots & w_{0r}\\
w_{10} & w_{11} & \ldots & w_{1r}\\
w_{20} & w_{21} & \ldots & w_{2r}\\
\vdots & \vdots & \, & \vdots\\
w_{r0} & w_{r1} & \ldots & w_{rr}
\end{array}
\mright].
$$
Now using \eqref{gw} and the fact that $\dim(\A)=d+1$, it follows $\rank(P)=d+1$. If $\rank(P)=r+1$ the result follows from Theorem \ref{gf}.
\hfill$\hbox{\rule{3pt}{6pt}}$
\medskip


\subsection{Algorithmic approach for deciding if $A_i$ is polynomial in $A$}
\label{gn}

In this subsection we give an algorithm which, for a given graph $\G$, decides whether  $A_i$ $(0\le i\le D)$  is a polynomial (not necessarily of degree $i$)  in $A$ or not. If the answer is in the affirmative, the algorithm also compute that polynomial. Note that this procedure can be seen as a refinement of Algorithm \ref{3d2}, since allows to decide if  $\G$ is distance-polynomial ($A_i\in \A$ for every $i=0,\ldots,D$).


\begin{algorithm}
\label{gi}
{\rm
Let $A$ denote the adjacency matrix of $\G$ with $d+1$ distinct eigenvalues and diameter $D$. Considering only the matrix $Z$ (from Definition \ref{gl}) we can determine which distance-$i$ matrix is a polynomial  in $A$ (see Example \ref{gg}).

\medskip
\noindent
{\bf Input:} The adjacency matrix $A$ of $\G$, or intersection diagrams around every vertex.\\
\noindent
{\bf Output:}  A polynomial $p_i$ such that $A_i=p_i(A)$ (if such a polynomial exists).
\begin{itemize}
\item[{\bf 1.}]
Using the adjacency matrix $A$ of $\G$ (or using intersection diagrams around every vertex), compute the vectors $\ww(y,z)$ for every $y,z\in X$ (see Definition \ref{gb} and Remark \ref{gm}).
\item[{\bf 2.}]
Find the matrices $[W\mid\tt]$, $[Z\mid\pp(t)]$, and the polynomials $p_i(t)$ $(0\le i\le d)$ from Definition~\ref{gl}.
\item[{\bf 3.}]
The columns of the matrices $W$ and $Z$ are indexed by the sets $\{R_0,R_1,\ldots,R_r\}$ (where $\R=\{R_0,R_1,\ldots,R_r\}$ is the walk-regular partition of $X\times X$). Let $R_{i_1},R_{i_2},\ldots,R_{i_k}$ denote the equivalence classes for which all pair of vertices in any $R_{i_h}$ $(0\le h\le k)$ are at the same distance. These relations represent the columns $i_h$ $(0\le h\le k)$ in $[W|\tt]$ and $[Z| \pp(t)]$. Let $p_{j_1},p_{j_2},\ldots,p_{j_m}$ denote the polynomials which have nonzero entry in the columns ${i_h}$ $(0\le h\le k)$ of $Z$.
\item[{\bf 4.}]
If the sum of the rows ${j_1},{j_2}, \ldots, {j_m}$ of $Z$ is a $(0,1)$-row vector for which the nonzero entry is only in columns $R_{i_1}, R_{i_2}, \ldots, R_{i_k}$, and vice versa, then the adjacency matrix $A_i$ is polynomial in $A$, and we have $A_i=p_{j_1}(A)+p_{j_2}(A)+\ldots+p_{j_m}(A)$. Otherwise, $A_i$ is not polynomial in $A$.
\end{itemize}
}\end{algorithm}

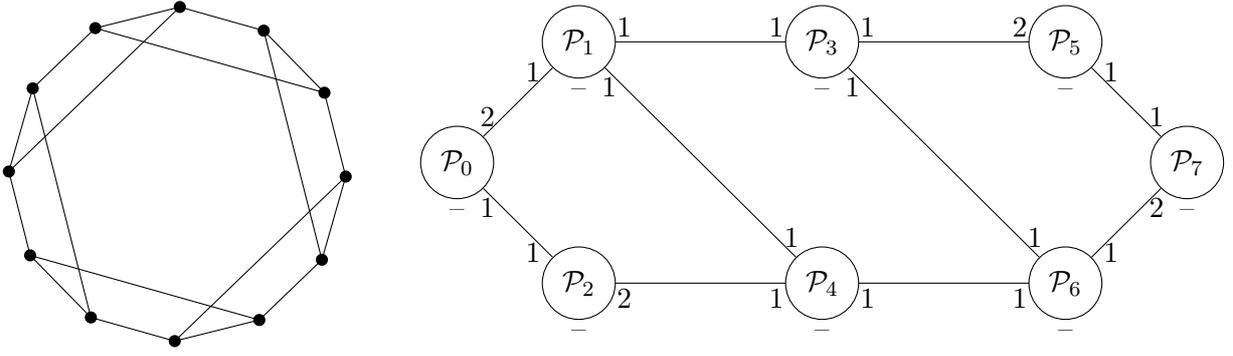
\begin{figure}[!ht]
{
\small
\begin{center}
\begin{tikzpicture}[scale=.4]
\fill (5.437550622520739,3.686217782649107) circle [radius=0.2];
\fill (6.221332839871632,6.4437684051698465) circle [radius=0.2];
\fill (5.521332839871632,9.223768405169846) circle [radius=0.2];
\fill (3.5251150572225267,11.281319027690586) circle [radius=0.2];
\fill (0.7675644347017866,12.065101245041479) circle [radius=0.2];
\fill (-2.0124355652982127,11.36510124504148) circle [radius=0.2];
\fill (-4.069986187818953,9.368883462392375) circle [radius=0.2];
\fill (-4.853768405169847,6.611332839871634) circle [radius=0.2];
\fill (-4.153768405169848,3.831332839871636) circle [radius=0.2];
\fill (-2.157550622520742,1.7737822173508944) circle [radius=0.2];
\fill (0.6,0.99) circle [radius=0.2];
\fill (3.38,1.69) circle [radius=0.2];
\draw  (-2.0124355652982127,11.36510124504148)-- (5.521332839871632,9.223768405169846);
\draw  (3.5251150572225267,11.281319027690586)-- (5.437550622520739,3.686217782649107);
\draw  (0.7675644347017866,12.065101245041479)-- (-4.853768405169847,6.611332839871634);
\draw  (-4.069986187818953,9.368883462392375)-- (-2.157550622520742,1.7737822173508944);
\draw  (6.221332839871632,6.4437684051698465)-- (0.6,0.99);
\draw  (3.38,1.69)-- (-4.153768405169848,3.831332839871636);
\draw  (-2.0124355652982127,11.36510124504148)-- (-4.069986187818953,9.368883462392375);
\draw  (-4.069986187818953,9.368883462392375)-- (-4.853768405169847,6.611332839871634);
\draw  (-4.853768405169847,6.611332839871634)-- (-4.153768405169848,3.831332839871636);
\draw  (-4.153768405169848,3.831332839871636)-- (-2.157550622520742,1.7737822173508944);
\draw  (-2.157550622520742,1.7737822173508944)-- (0.6,0.99);
\draw  (0.6,0.99)-- (3.38,1.69);
\draw  (3.38,1.69)-- (5.437550622520739,3.686217782649107);
\draw  (5.437550622520739,3.686217782649107)-- (6.221332839871632,6.4437684051698465);
\draw  (6.221332839871632,6.4437684051698465)-- (5.521332839871632,9.223768405169846);
\draw  (5.521332839871632,9.223768405169846)-- (3.5251150572225267,11.281319027690586);
\draw  (3.5251150572225267,11.281319027690586)-- (0.7675644347017866,12.065101245041479);
\draw  (0.7675644347017866,12.065101245041479)-- (-2.0124355652982127,11.36510124504148);
\end{tikzpicture}
\qquad
\begin{tikzpicture}[scale=.8]
\draw (1,3) circle [radius=0.6];
\node  at (1,3) {{\small$\P_0$}};
\node  at (1.5,3.75) {$2$};
\node  at (1.5,2.25) {$1$};
\node  at (1,2.2) {{\small --}};

\draw (3,1) circle [radius=0.6];
\node  at (3,1) {{\small$\P_2$}};
\node  at (2.25,1.5) {$1$};
\node  at (3.75,0.75) {$2$};
\node  at (3,0.2) {{\small --}};

\draw (3,5) circle [radius=0.6];
\node  at (3,5) {{\small$\P_1$}};
\node  at (2.25,4.5) {$1$};
\node  at (3.5,4.25) {$1$};
\node  at (3.75,5.25) {$1$};
\node  at (3,4.2) {{\small --}};

\draw (7,1) circle [radius=0.6];
\node  at (7,1) {\small$\P_4$};
\node  at (6.25,0.75) {$1$};
\node  at (6.5,1.75) {$1$};
\node  at (7.75,0.75) {$1$};
\node  at (7,0.2) {{\small --}};

\draw (7,5) circle [radius=0.6];
\node  at (7,5) {{\small$\P_3$}};
\node  at (6.25,5.25) {$1$};
\node  at (7.75,5.25) {$1$};
\node  at (7.5,4.25) {$1$};
\node  at (7,4.2) {{\small --}};

\draw (11,1) circle [radius=0.6];
\node  at (11,1) {{\small$\P_6$}};
\node  at (10.25,0.75) {$1$};
\node  at (10.5,1.75) {$1$};
\node  at (11.75,1.5) {$1$};
\node  at (11,0.2) {{\small --}};

\draw (11,5) circle [radius=0.6];
\node  at (11,5) {\small$\P_5$};
\node  at (10.25,5.25) {$2$};
\node  at (11.75,4.5) {$1$};
\node  at (11,4.2) {{\small --}};

\draw (13,3) circle [radius=0.6];
\node  at (13,3) {\small$\P_7$};
\node  at (12.5,3.75) {$1$};
\node  at (12.5,2.25) {$2$};
\node  at (13,2.2) {{\small --}};
\draw (1.43,2.57)--(2.57,1.43);
\draw (1.43,3.43)--(2.57,4.57);

\draw (3.6,1)--(6.4,1);
\draw (3.6,5)--(6.4,5);
\draw (3.43,4.57)--(6.57,1.43);

\draw (7.6,1)--(10.4,1);
\draw (7.43,4.57)--(10.57,1.43);
\draw (7.6,5)--(10.4,5);

\draw (11.43,1.43)--(12.57,2.57);
\draw (11.43,4.57)--(12.57,3.43);
\end{tikzpicture}
\caption{`Chordal ring' $(12,4)$ and its intersection diagram. This graph has the same intersection diagram around every vertex and adjacency algebra $\A$ is not closed with respect to Hadamard product. If $\R=\{R_0,R_1,\ldots,R_7\}$ is the walk-regular partition and if $F_i$ $(0\le i\le 7)$ are adjacency matrices of $R_i$ $(0\le i\le 7)$, then for a fixed vertex $x$ of $\G$ we have $\P_i=\{z| (F_i)_{xz}=1\}$ $(0\le i\le 7)$.}
\label{gg}
\end{center}
}
\end{figure}

\begin{example}{\rm
\label{gh}
Assume that $\G$ is the graph from Figure \ref{gg}. Using the intersection diagram we can compute the adjacency matrix $B\in\Mat_{8\times 8}(\CC)$ of intersection diagram, and using $B$, we can compute the numbers $w_{ij}$ from Definition \ref{gl} (for example, a number $(B^\ell)_{\P_3,\P_0}$ is the number $w_{\ell 3}$ $(0\le\ell\le 7)$). Since we do not know the number of distinct eigenvalues, using Corollary \ref{gs} we know that $\G$ will not have more then $8$ of them. So we can compute the matrices $W$ and $Z$ with $8$ rows and $8$ columns. We have
$$
\underbrace{\left(\begin{array}{cccccccc|c}
1 & 0 & 0 & 0 & 0 & 0 & 0 & 0 & 1\\
0 & 1 & 1 & 0 & 0 & 0 & 0 & 0 & t\\
3 & 0 & 0 & 1 & 2 & 0 & 0 & 0 & t^2\\
0 & 6 & 7 & 0 & 0 & 2 & 3 & 0 & t^3\\
19 & 0 & 0 & 11 & 16 & 0 & 0 & 8 & t^4\\
0 & 46 & 51 & 0 & 0 & 30 & 35 & 0 & t^5\\
143 & 0 & 0 & 111 & 132 & 0 & 0 & 100 & t^6\\
0 & 386 & 407 & 0 & 0 & 322 & 343 & 0 & t^7
\end{array}\right)}_{=[W|\tt]}
\stackrel{\rm row}{\sim}
\underbrace{
\left(\begin{array}{cccccccc|c}
1 & 0 & 0 & 0 & 0 & 0 & 0 & 0 & p_0(t)\\
0 & 1 & 0 & 0 & 0 & 0 & -1 & 0 & p_1(t)\\
0 & 0 & 1 & 0 & 0 & 0 & 1 & 0 & p_2(t)\\
0 & 0 & 0 & 1 & 0 & 0 & 0 & 0 & p_3(t)\\
0 & 0 & 0 & 0 & 1 & 0 & 0 & 0 & p_4(t)\\
0 & 0 & 0 & 0 & 0 & 1 & 1 & 0 & p_5(t)\\
0 & 0 & 0 & 0 & 0 & 0 & 0 & 1 & p_6(t)\\
0 & 0 & 0 & 0 & 0 & 0 & 0 & 0 & *
\end{array}\right)
}_{=[Z\mid\pp(t)]}
$$
where polynomials $p_i(t)$ $(0\le i\le 6)$ are
$$
p_0(t)=1,\qquad
p_1(t)=\frac{1}{10}t^5 - \frac{3}{2} t^3 + \frac{27}{5} t,\qquad
p_2(t)=-\frac{1}{10} t^5 + \frac{3}{2}t^3 - \frac{22}{5} t,\qquad
$$
$$
p_3(t)=\frac{2}{15}t^6 - \frac{5}{3}t^4 + \frac{68}{15}t^2 - 1,\qquad
p_4(t)=-\frac{1}{15}t^6 + \frac{5}{6}t^4 - \frac{53}{30}t^2 - 1,\qquad
p_5(t)=\frac{1}{20}t^5 - \frac{1}{4}t^3 - \frac{4}{5}t,
$$
$$
p_6(t)=-\frac{1}{20} t^6 + \frac{3}{4} t^4 - \frac{27}{10} t^2 + 1.
$$
Since $\rank(W)=7$, $\G$ has $7$ distinct eigenvalues, which imply that the polynomial $p_7(t)$ is not important. Note that $A_0=p_0(A)$, $A_1=p_1(A)+p_2(A)$, $A_2=p_3(A)+p_4(A)$, $A_3=p_5(A)$ and $A_4=p_6(A)$. Therefore, every distance-$i$ matrix can be write as a polynomial in $A$ and $\sum_{i=0}^6 p_i(t)$ is the Hoffman polynomial. Thus, $\G$ is not a quotient-polynomial graph.
}\end{example}


\section{Some characterizations of quotient polynomial graphs}
\label{4a}

In this section we prove Theorems  \ref{1d} and  \ref{1h}.
First we prove Theorem \ref{1d}.
(The adjacency algebra of $\G$ is closed under Hadamard product if and only if $\G$ is a quotient-polynomial graph).

\medskip
\noindent
{\bf Proof of Theorem \ref{1d}.}
The proof of this claim follows from \cite[Theorem 4.1]{FQ}. Here we give an alternative proof for completeness and clarity.

Assume that $\G$ is a quotient-polynomial graph. Let $F_i$ $(0\le i\le d)$ denote the adjacency matrix of the equivalence class $R_i$ $(0\le i\le d)$ of a walk-regular partition $\R=\{R_0,R_1,\ldots,R_d\}$ of $X\times X$.
By definition, $\{I=F_0,F_1,\ldots,F_d\}$ is a linearly independent set such that $F_i\circ F_j=\delta_{ij}F_i$, and $\sum_{i=0}^d F_i=J$. Moreover since $F_i\in\A$ we have $\span\{F_0,F_1,\ldots,F_d\}\subseteq \A$. Thus, the vector space $\A$ is closed under both ordinary and Hadamard multiplication.

Conversely, assume that the vector space $\A$ is closed under both ordinary and Hadamard multiplication. By Theorem \ref{1c}, since $\G$ is a regular graph, the algebra $\A$ has the standard basis $\{I=F_0,F_1,\ldots,F_d\}$. Then, there exists scalars $\alpha_{ij}$ $(0\le i,j\le d)$ such that
\begin{equation}
\label{4b}
A^\ell = \sum_{j=0}^d \alpha_{\ell j} F_j \qquad (0\le\ell\le d).
\end{equation}
Now, by \eqref{4b}, if $u,v,y,z\in X$ are vertices such that $(F_i)_{uv}=1$ and $(F_i)_{yz}=1$ $(0\le i\le d)$, then the number of walks of length $\ell$ from $u$ to $v$, is equal to the number of walks of length $\ell$ from $y$ to $z$ $(0\le\ell\le d)$. This implies that the matrices $F_i$ correspond to the basis relations $R_i$ $(0\le i\le d)$, and that $\R=\{R_0,R_1,\ldots,R_d\}$ is a walk-regular partition of $X\times X$. Since $F_i\in\A$ the result follows.
\hfill$\hbox{\rule{3pt}{6pt}}$
\medskip

\begin{corollary}
Let $\G$ be a graph with adjacency matrix $A$ and $d+1$ distinct eigenvalues. If some matrix in $\{A^2,\ldots,A^d\}$ has more than $d+1$ distinct entries, then $\A$ is not closed under Hadamard product. 
\end{corollary}



Now we prove Theorem \ref{1h}.
(A regular graph $\G$ with diameter $2$ and $4$ distinct eigenvalues is quotient-polynomial if and only if either
any two nonadjacent (respectively, adjacent) vertices have a constant number of common neighbours, and the number of common neighbours of any two adjacent (respectively, nonadjacent) vertices takes precisely two values.)

The proof can be seen as very nice application of the walk-regular partition from Section \ref{go}.

\bigskip
\noindent
{\bf Proof of Theorem \ref{1h}.}
\begin{itemize}
\item[$(\Rightarrow)$]
Assume that the vector space $\A=\span\{I,A,A^2,A^3\}$ is closed under Hadamard multiplication. By Theorem \ref{1c}, $\A$ has the standard basis $\{F_0,F_1,F_2,F_3\}$ consisting of $\circ$-idempotents. For every $\ell$ $(0\le\ell\le 3)$ there exist scalars $\alpha_{\ell i}$ $(0\le i\le 3)$ such that
$$
A^\ell=\alpha_{\ell 0} F_0 +\alpha_{\ell 1} F_1 +
\alpha_{\ell 2} F_2 + \alpha_{\ell 3} F_3.
$$
This implies that if $(F_i)_{yz}\ne 0$ then $(A^\ell)_{yz}=\alpha_{\ell i}$. Thus, for every $y,z,u,v\in X$, if $(F_i)_{yz}\ne 0$ and $(F_i)_{uv}\ne 0$ then
$$
(A^\ell)_{yz}=(A^\ell)_{uv}
\qquad
(0\le \ell\le 3).
$$
Now, we can obtain a walk-regular partition $\R=\{R_0,R_1,R_2,R_3\}$ (see Definition \ref{ga}) in the following way:
$$
(z,y)\in R_i
\qquad\Leftrightarrow\qquad
(F_i)_{zy}\ne 0
\qquad(0\le i\le 3).
$$
By Lemma \ref{gc}, all pairs of vertices in a given $R_i$ are at the same distance. This implies that if $(F_i)_{zy}\ne 0$ and $(F_i)_{uv}\ne 0$ then $\dist(z,v)=\dist(u,v)$ for every $z,y,u,v\in X$. Permute indices of the set $\{F_0,F_1,F_2,F_3\}$ so that $F_0=I$, and, for any $i\le j$ and $(F_i)_{zy}\ne 0$, $(F_j)_{uv}\ne 0$ we have $\dist(z,y)\le \dist(u,v)$. Since $\G$ is a graph of diameter $2$, $(F_3)_{zy}\ne 0$ implies $\dist(z,y)=2$. Since there exist scalars $\beta_i$ $(0\le i\le 3)$ such that
$$
A=\beta_0 I + \beta_1 F_1 + \beta_2 F_2+ \beta_3 F_3
$$
and since $A$ is $(0,1)$-matrix, we have $\beta_0=0$ and only one of the following two cases are possible: $A=F_1+F_2$ or $A=F_1$.

\medskip
\begin{itemize}
\item[{\it Case 1.}]
Assume that $A=F_1+F_2$. This yields $F_3=A_2$. Now, it is not hard to see that there exists scalars $k,\lambda_1,\lambda_2,\mu$ such that
$$
A^2= k I + \lambda_1 F_1 + \lambda_2 F_2 + \mu F_3,
$$
and the result follows.

\item[{\it Case 2.}]
Assume that $A=F_1$. This yields $F_2+F_3=A_2$. Now, there exists scalars $k,\lambda,\mu_1,\mu_2$ such that
$$
A^2= k I + \lambda F_1 + \mu_1 F_2 + \mu_2 F_3,
$$
and the result follows.
\end{itemize}

\item[$(\Leftarrow)$]
Assume that $\G$ has the property (i), that is  any two vertices at distance two have exactly $\mu$ common neighbours, and for every adjacent $x,y\in X$ we have $|\G(x)\cap\G(y)|\in\{\lambda_1,\lambda_2\}$. Define the matrices $\{F_0,F_1,F_2,F_3\}$ as $F_0:=I$, $F_1+F_2=A$ where
$$
(F_1)_{xy}=1
\qquad\mbox{if and only if}\qquad
\dist(x,y)=1 \mbox{ and } |\G(x)\cap\G(y)|=\lambda_1
\qquad (x,y\in X),
$$
and let $F_3=A_2$. Since $\G$ is regular $J\in\A$. Note that $I+A+A_2=J$ yields $A_2\in\A$, and with that $F_3\in\A$. Let $k$ denote valency  of $\G$. Computing $A^2$ we have
$$
A^2=kI+\lambda_1F_1+\lambda_2F_2+\mu A_2=
kI+\lambda_1F_1+\lambda_2(A-F_1)+\mu A_2
$$
which yields $F_1\in\A$. Since $F_2=A-F_1$ we also have $F_2\in\A$. By construction the set $\{F_0,F_1,F_2,F_3\}$ is linearly independent set consisting of $\circ$-idempotents. Thus we showed that $\span\{F_0,F_1,F_2,F_3\}\subseteq\A$. The result follows.

If we assume that $\G$ has the property (ii), the proof is similar as above (consider the set of $(0,1)$-matrices $\{I,A,F_2,F_3\}$ where $F_2+F_3=A_2$, and $(F_2)_{xy}=1$ if and only if $\dist(x,y)=2$ and $|\G(x)\cap\G(y)|=\mu_1$).
\hfill$\hbox{\rule{3pt}{6pt}}$
\end{itemize}

\medskip
The two families of graphs from Theorem \ref{1h} are in fact a subfamily of Deza graphs (see \cite{EF}). Note that, if $\G$ is a graph for which property (i) of Theorem \ref{1h} holds, then the distance-$2$ matrix of $\G$ is the adjacency matrix of $\ol{\G}$ (complement of $\G$, which have the property that any two adjacent vertices have a constant number of common neighbours, and the number of common neighbours of any two nonadjacent vertices takes precisely two values). With this in mind, it follows a result of Van Dam from \cite{ED}:

\begin{theorem}[{{\rm\cite[Theorem 5.1]{ED}}}]
\label{4f}
Let $\G$ be a connected regular graph with four distinct eigenvalues and diameter $2$. Then $\G$ is one of the relations of a $3$-class association scheme if and only if any two adjacent vertices have a constant number of common neighbours, and the number of common neighbours of any two nonadjacent vertices takes precisely two values.
\end{theorem}


\section{The existence of an idempotent generator}
\label{5a}

In this section we prove Theorem \ref{1e}
(a given $F\in \{F_0,F_1,\ldots,F_d\}$ has $d+1$ distinct eigenvalues if and only if $\langle F_0,F_1,\ldots,F_d \rangle=\langle I,F,\ldots,F^d\rangle$).

\medskip
\noindent
{\bf Proof of Theorem \ref{1e}.}
We already know that, for any real symmetric matrix $B$ with $s+1$ distinct eigenvalues, the set $\{I,B,\ldots,B^s\}$ is a basis of the algebra $\{p(B)\mid p\in\RR[t]\}$.


\begin{itemize}
\item[$(\Leftarrow)$]
Assume that $\A=\span\{I,F,\ldots,F^d\}$. This yield that $\{I,F,\ldots,F^d\}$ is also a basis of $\A$, that is, it is maximal linearly independent set. Thus $F$ have $d+1$ distinct eigenvalues.

\item[$(\Rightarrow)$]
 Now assume that $F$ has $d+1$ distinct eigenvalues, and let $\F$ denote the algebra generated by the set $\{I,F^1,\ldots,F^d\}$. Since $\{I,F_1,\ldots,F_d\}$ is a basis of $\A$ we have that $F^i\in\A$ for every $i\in\NN$. This yields $\F\subseteq\A$, that is $\dim(\F)\le d+1$. Now since $F$ has $d+1$ distinct eigenvalues, $\dim(\F)=d+1$, and the result follows.
\hfill$\hbox{\rule{3pt}{6pt}}$
\end{itemize}

\begin{example}{\rm
Let $\G$ denote the bipartite $2$-walk-regular graph with diameter $4$ and $6$ distinct eigenvalues from \cite[Theorem 2]{JK}. By such a theorem, $\G$ generates an association scheme with $5$ classes. Let $\{A_0,A_1,\ldots,A_5\}$ denote the adjacency matrices of this association scheme. Considering its first eigenmatrix $P$ \cite[Section 3]{JK}, we can conclude that $A_1$ and $A_3$ have 6 different eigenvalues. Thus, both of these matrices generate the algebra $\A$ of $\G$, which is closed under Hadamard multiplication.
}\end{example}


\section{Further directions}
\label{6a}

Let $\G$ denote a quotient-polynomial graph with vertex set $X$, $d+1$ distinct eigenvalues, and let $\{I,F_1,\ldots,F_d\}$ be the standard basis of the adjacency algebra $\A$ of $\G$.

Since $\{E_0,E_1,\ldots,E_d\}$ is also a basis of $\A$, there exist numbers $q^h_{ij}$ such that
\begin{equation}
\label{6b}
E_i\circ E_j=\frac{1}{|X|}\sum_{h=0}^d q^h_{ij} E_h
\qquad(0\le i,j\le d).
\end{equation}
The numbers $q^h_{ij}$ are called the {\it Krein parameters} for $\G$ with respect to the ordering $E_0,E_1,\ldots,E_d$ of its basis of primitive idempotents. An ordering $E_0,E_1,\ldots,E_d$ is a {\it cometric $($$Q$-polynomial\/$)$ ordering} if the following conditions are satisfied:
\begin{enumerate}[label=\rm(Q\arabic*)]
\item
$q^h_{ij}=0$ whenever any one of the indices $i,j,h$ exceed the sum of the remaining two, and
\item
$q^h_{ij}>0$ when $0\le i,j,h\le d$ and any one of the indices equals the sum of the remaining two.
\end{enumerate}
We say that $\G$ is a {\it cometric} (or {\it $Q$-polynomial}) quotient-polynomial graph when such an ordering exists. In the future, we plan to study algebraic and combinatorial properties of cometric quotient-polynomial graphs. This $Q$-polynomial concept is taken from the theory of commutative association schemes. A good introduction to the topic of $Q$-polynomial structures for association schemes and distance-regular graphs can be found in \cite{GD}. For a new technique (and approach) about computations in Bose-Mesner algebras, which also deals with $Q$-polynomial case, we recommend \cite[Section 3]{WMS}.

Fix a ``base vertex'' $x\in X$. For each $i$ $(0\le i\le D)$ let $F^*_i=F^*_i(x)$ denote the diagonal matrix in $\Mat_X(\CC)$ with $(y,y)$-entries $(F^*_i)_{yy}=(F_i)_{xy}$. The {\it Terwilliger} (or {\it subconstituent}) algebra $\T=\T(x)$ of $\G$ with respect to $x$ is the subalgebra of $\Mat_X(\CC)$ generated by $\{I,F_1,\ldots,F_d,F^*_0,F^*_1,\ldots,F^*_D\}$. By a $T$-{\em module} we mean a subspace $\W$ of $\V=\CC^X$ such that $B\W \subseteq \W$ for all $B \in \T$. Let $\W$ denote a $T$-module. Then $\W$ is said to be {\em irreducible} whenever $\W$ is nonzero and $\W$ contains no $T$-modules other than $0$ and $\W$. In the future we plan to study irreducible $T$-modules of quotient-polynomial graph $\G$. This $T$-module concept is also taken from the theory of commutative association schemes \cite{T1, T3, T4}. For most recent research on the use of Terwilliger algebra in the study of $P$-polynomial association schemes (that is, using the Terwilliger algebra to study distance-regular graphs) see \cite{CW, MM1, MM2, MMP1, MMP2, MSq, MJV, MX, PS}.

Another possible line of research would be the study of `pseudo-quotient polynomial graphs', defined by using weighted regular partitions,
see \cite{Fei}.


 \end{document}